\documentclass[11pt]{article}
\usepackage{graphicx,amsthm,graphicx,fancyhdr,mathrsfs}
\usepackage{amsfonts}
\usepackage{amsmath}
\usepackage{amssymb}
\usepackage{url}
 \usepackage[colorlinks=true,citecolor=blue]{hyperref}
\usepackage{fancyhdr}
\usepackage{indentfirst}
\usepackage{enumerate}
\usepackage{amsthm}
\usepackage{color}
\usepackage{comment}
\usepackage{dsfont}
\usepackage{natbib}
\usepackage[misc]{ifsym}
\usepackage[toc,page]{appendix}

\def\d{\mathrm{d}}
\def\laweq{\buildrel \d \over =}

\newcommand{\var}{\mathrm{Var}}
\newcommand{\cov}{\mathrm{Cov}}

\newcommand{\E}{\mathbb{E}}
\newcommand{\R}{\mathbb{R}}

\newcommand{\N}{\mathbb{N}}
\newcommand{\p}{\mathbb{P}}

\newcommand{\id}{\mathds{1}}

\newcommand{\XX}{\mathbf X}

\renewcommand{\(}{\left(}
\renewcommand{\)}{\right)}

\renewcommand{\ge}{\geqslant}
\renewcommand{\le}{\leqslant}
\renewcommand{\geq}{\geqslant}
\renewcommand{\leq}{\leqslant}
\renewcommand{\epsilon}{\varepsilon}

\usepackage{array}
\newcommand{\PreserveBackslash}[1]{\let\temp=\\#1\let\\=\temp}
\newcolumntype{C}[1]{>{\PreserveBackslash\centering}p{#1}}
\newcolumntype{R}[1]{>{\PreserveBackslash\raggedleft}p{#1}}
\newcolumntype{L}[1]{>{\PreserveBackslash\raggedright}p{#1}}

\usepackage{booktabs,array}

\newcount\rowc

\makeatletter
\def\ttabular{%
\hbox\bgroup
\let\\\cr
\def\rulea{\ifnum\rowc=\@ne \hrule height 1.3pt \fi}
\def\ruleb{
\ifnum\rowc=1\hrule height 1.3pt \else
\ifnum\rowc=6\hrule height \heavyrulewidth
   \else \hrule height \lightrulewidth\fi\fi}
\valign\bgroup
\global\rowc\@ne
\rulea
\hbox to 10em{\strut \hfill##\hfill}%
\ruleb
&&%
\global\advance\rowc\@ne
\hbox to 10em{\strut\hfill##\hfill}%
\ruleb
\cr}
\def\endttabular{%
\crcr\egroup\egroup}

\theoremstyle{plain}
\newtheorem{theorem}{Theorem}

\newtheorem{lemma}{Lemma}
\newtheorem{proposition}{Proposition}
\theoremstyle{definition}
\newtheorem{definition}{Definition}
\newtheorem{example}{Example}

\newtheorem{remark}{Remark}

\usepackage{tikz}

\usepackage[compact]{titlesec}

\makeatletter
\DeclareRobustCommand{\bsquare}{%
  \mathop{\vphantom{\sum}\mathpalette\bigstar@\relax}\slimits@
}

\newcommand{\bigstar@}[2]{%
  \vcenter{%
    \sbox\z@{$#1\sum$}%
    \hbox{\resizebox{.9\dimexpr\ht\z@+\dp\z@}{!}{$\m@th\dsquare$}}%
  }%
}
\makeatother

\newcommand{\dsquare}{\mathop{  \square} \displaylimits}

\usepackage[onehalfspacing]{setspace}

\begin{document}

\title{Joint mixability and notions of negative dependence}

\author{Takaaki Koike\thanks{
Graduate School of Economics, Hitotsubashi University, Japan.
\Letter~{\texttt{takaaki.koike@r.hit-u.ac.jp}}
} \and Liyuan Lin\thanks{Department of Statistics and Actuarial Science, University of Waterloo, Canada. \Letter~{\texttt{l89lin@uwaterloo.ca}}
} \and  Ruodu Wang\thanks{Department of Statistics and Actuarial Science, University of Waterloo, Canada.
\Letter~{\texttt{wang@uwaterloo.ca}}
}
}
\maketitle

\begin{abstract}
A joint mix is a random vector with a constant component-wise sum.
The dependence structure of a joint mix minimizes some common objectives such as the variance of the component-wise sum, and it is regarded as a concept of extremal negative dependence.
In this paper, we explore the connection between the joint mix structure and popular notions of negative dependence in statistics, such as negative correlation dependence, negative orthant dependence and negative association.
A joint mix is not always negatively dependent in any of the above senses, but some natural classes of joint mixes are.
We derive various necessary and sufficient conditions for a joint mix to be negatively dependent, and study the compatibility of these notions.
For identical marginal distributions, we show that a negatively dependent joint mix solves a multi-marginal optimal transport problem for quadratic cost under a novel setting of uncertainty. 
Analysis of this optimal transport problem with heterogeneous marginals reveals a trade-off between negative dependence and the joint mix structure.
\medskip
\\
\noindent   \emph{Keywords:}
Joint mixability; negative dependence; optimal transport; extreme dependence; uncertainty
\\
\emph{MSC 2000}: Primary: 62H05; Secondary: 	49Q22, 64H20, 91B05
\end{abstract}

\section{Introduction}\label{sec:intro}

Dependence among multiple sources of randomness has always been an active topic in operations research,  statistics,   transport theory, economics, and finance; see \cite{DDGK05}, \cite{J14}, \cite{R13}, \cite{MFE15} and \cite{G16}  for standard textbook treatment in different fields, and the recent work \cite{BLLW20} for   relevant examples in operations research.
In contrast to positive, which received much attention in the literature, considerably fewer studies are found on negative dependence, partially due to its more complicated mathematical nature. 
For a review and historical account on
extremal positive and negative dependence concepts, we refer to \cite{PW15}. 

In the past decade,  the notion of joint mixability proposed by \cite{WPY13}, which generalizes complete mixability~\citep{WW11}, has been shown useful for solving many optimization problems involving the dependence of multiple risks.
In particular, joint mixability is essential to worst-case bounds on Value-at-Risk and other risk measures under dependence uncertainty \citep{PR13, EPR13,
 BJW14}, as well as bottleneck assignment and scheduling problems \citep{CY84,H84, H15, BBV18}.

Joint mixability concerns, for given marginal distributions, the existence of a  random vector which has a constant component-wise sum. Such a random vector is called a \emph{joint mix} supported by the given marginal distributions, and it represents a very simple concept of dependence.
A joint mix is commonly regarded as a notion of extremal negative dependence; see the review of \cite{PW15}.
The reason why a joint mix represents  negative dependence   is that it minimizes many objectives which are maximized by comonotonicity. For instance, for fixed marginal distributions of the risks, comonotonicity maximizes the variance, the stop-loss premium, and the Expected Shortfall (ES) of the sum of the risks,
whereas a joint mix, if it exists, minimizes these quantities; see e.g., \cite{R13}. As such, joint mixability is seen as the safest dependence structure, as long as risk aggregation is concerned~\citep{EPRWB14}.

Although a joint mix has been treated as a concept of negative dependence, it remains unclear whether it is consistent with   classic notions of negative dependence in statistics.
Popular notions of negative dependence include 
 negative correlation dependence (NCD), 
 negative orthant dependence (NOD; \citealp{BSS82,L66}) and negative association (NA; \citealp{AS81, JP83}).
The connection between joint mixes and these negative dependence concepts is the main object that we address in this paper.
We obtain some necessary and sufficient conditions for a joint mix to be NOD or NA in Section \ref{sec:JM}.
Some characterization results are obtained in Section~\ref{sec:elliptical} for the class of elliptical distributions.
In particular, among all elliptical classes, only the Gaussian family supports NOD joint mixes of any dimension. 


 Since a joint mix may be either  negatively dependent or not, a natural question is whether there are special features of negatively dependent joint mixes which are useful in applications. 
For this question, we consider a multi-marginal optimal transport problem under uncertainty on the set of components.  A few optimality results on negatively dependent joint mixes are obtained, and they demonstrate an interesting interplay between joint mixes and negative dependence.  
In particular, for the special case of quadratic cost,  we show that the optimizer has to be an NCD JM in some settings.
This is the topic of Section \ref{sec:otp}.


The study of joint mixability was originally motivated by questions in risk management and operations research, and it has a strong connection to the theory of multi-marginal optimal transport~\citep{S15, P15} and variance reduction in random sampling~\citep{CM01,CM05}; see also our Section \ref{sec:otp}.
Recently, there is a growing spectrum of applications of joint mixability outside the above fields, including multiple hypothesis testing~\citep{VWW22}, wireless communications~\citep{BJ20}, labor market matching~\citep{BTZ21}, and resource allocation games~\citep{PRG22}.
Results in this paper connect the two topics of joint mixability and negative dependence, allowing us to bring tools from one area to the other.


This paper is organized as follows.
Section~\ref{sec:notions:negative:dependence} introduces the concepts of negative dependence and joint mixability, and summarizes their basic interrelationships.
Section~\ref{sec:JM} explores conditions for joint mixes to be negatively dependent. Two  results on  decompositions of joint mixes into negatively dependent ones are also obtained.
Section~\ref{sec:otp} studies a multi-marginal optimal transport problem 
 as an application of negatively dependent joint mixes.
Section~\ref{sec:elliptical}  studies joint mixes within the elliptical family, and we obtain a new characterization of the Gaussian family as the only one supporting a negatively dependent elliptical distribution for every dimension. Section~\ref{sec:conclusion} concludes the paper with  some open questions and potential directions for future research.
All the proofs are deferred to Appendix~\ref{sec:proofs}.


\section{Notions of negative dependence}\label{sec:notions:negative:dependence}

In this section we recall a few classic notions of negative dependence.  Throughout, denote by $[n]=\{1,\dots,n\}$ and $\mathbf 1_n$ the $n$-vector with all components being $1$; the vector $\mathbf 0_n$ is defined analogously.
All inequalities and equalities between (random) vectors are component-wise.
For an $n$-dimensional random vector $\mathbf{X}=(X_1,\dots,X_n)$, denote by ${\mathbf X}^\perp=(X_1^\perp,\dots,X_n^\perp)$ a random vector with independent components such that $X_i \laweq X_i^\perp$, $i\in [n]$, where $\laweq$ stands for equality in distribution. 
For a set $A\subseteq [n]$, we denote by
 $\mathbf X_A=(X_k)_{k\in A}$.
 A function $\psi:\R^n \rightarrow \R$ is called \emph{supermodular} if $\psi(\mathbf x \wedge \mathbf y)+\psi(\mathbf x \vee \mathbf y)\geq \psi(\mathbf x)+\psi(\mathbf y)$ for all $\mathbf x,\mathbf y\in \R^n$, where $\mathbf x \wedge \mathbf y$ and  $\mathbf x \vee \mathbf y$ are the component-wise minimum and maximum of $\mathbf x$ and $\mathbf y$, respectively.

\begin{definition}
Let $\mathbf X=(X_1,\dots,X_n)$  be an $n$-dimensional random vector.
\begin{enumerate}[(i)]
    \item $\mathbf X $  is   \emph{negative correlation dependent (NCD)} if  $
\cov(X_i,X_j)\le 0
$
for all $i,j\in [n]$ with $i\ne j$.
    \item  $\mathbf X $ is \emph{negative upper orthant dependent (NUOD)} if $\p(\mathbf{X}  >  \mathbf {t} )\le \p(\mathbf{X}^\perp  >  \mathbf t)$ for all $\mathbf t \in \R^n$; $\mathbf X$ is \emph{negative lower orthant dependent (NLOD)} if  $\p(\mathbf{X}  \le  \mathbf {t} )\le \p(\mathbf{X}^\perp  \le  \mathbf t)$ for all $\mathbf t \in \R^n$.
If $\mathbf X$ is both NLOD and NUOD, then it is \emph{negative orthant dependent (NOD)}.
    \item $\mathbf X$ is \emph{negative supermodular dependent (NSD)} if $\mathbb E[\psi(\mathbf X)]\leq \mathbb E[\psi(\mathbf X^\perp)]$ for all supermodular functions $\psi:\R^n \rightarrow \R$ such that the expectations exist. 
    \item  $\mathbf X$  is   \emph{negatively associated (NA)} if  
\begin{align}\label{eq:def-na}
\text{Cov}(f(\mathbf X_A),g(\mathbf X_B)) \leq 0,
\end{align} 
for any disjoint subsets $A,B \subseteq [n]$ and any real-valued, coordinate-wise increasing functions $f$ and $g$ 
such that $f(\mathbf X_A)$ and $g(\mathbf X_B)$ have finite second moments.
\item $\mathbf X$ is \emph{counter-monotonic (CT)}
if each pair of its component $(X_i,X_j)$ for $i\ne j$
satisfies $(X_i,X_j)=(f(Z),-g(Z))$ almost surely (a.s.) for some random variable $Z$ and increasing functions $f,g$.
\item $\mathbf X$ is a \emph{joint mix} (abbreviated as ``$\mathbf X$ is JM") if $\sum_{i=1}^n X_i=c $ a.s.~for some constant $c\in \R$.  
\end{enumerate}
\end{definition}

All abbreviations introduced in this section are also used as nouns to represent the corresponding dependence concept. 
The next definition concerns properties of the marginal distributions that allow for JM random vectors.
\begin{definition}
An $n$-tuple $(F_1,\dots,F_n)$ of  distributions on $\R$ is
called \emph{jointly mixable} if there exists a  joint mix $\mathbf X=(X_1,\dots,X_n)$ such that $X_i \sim F_i$, $i \in [n]$. The constant $c=\sum_{i=1}^n X_i$ is called a \emph{center} of $\mathbf X$.
In this case, we also say that \emph{$(F_1,\dots,F_n)$ supports a joint mix $\mathbf{X}$.}
A distribution $F$ is called $n$-\emph{completely mixable} if the $n$-tuple $(F,\dots,F)$ is jointly mixable. 
\end{definition}

The following implications hold between the above concepts of negative dependence. These implications are either checked directly by definition or shown in the literature, e.g., \cite{JP83}, \cite{CV04} and \cite{LLW23}. 
The case of   $n\ge 3$ is different from the case $n=2$.
\begin{align}\label{eq:chain-2}&  n=2 :~~~~ \mbox{JM} \Longrightarrow  \mbox{CT}   \Longrightarrow \mbox{NA}    \Longleftrightarrow \mbox{NSD}   \Longleftrightarrow \mbox{NOD} \Longleftrightarrow \mbox{NUOD} \Longleftrightarrow \mbox{NLOD} \Longrightarrow \mbox{NCD}; \\
\label{eq:chain} 
&\mbox{general $n$}:~~~~~~~~   \mbox{CT}   \Longrightarrow \mbox{NA}    \Longrightarrow \mbox{NSD}   \Longrightarrow \mbox{NOD} \Longrightarrow \mbox{NUOD or NLOD}\Longrightarrow \mbox{NCD} .
\end{align}
All one-direction implications  in \eqref{eq:chain-2} and \eqref{eq:chain} are strict for $n \ge 3$; see \cite{ASB13} for some examples.
In contrast to the case $n=2$ in \eqref{eq:chain-2}, JM no longer implies any of the properties in \eqref{eq:chain}. This can be observed by the following properties of Gaussian random vectors.

\begin{proposition}\label{prop:r1-G-1}
Let $\mathbf X\sim \operatorname{N}_n(\boldsymbol{\mu},\Sigma)$ be a Gaussian random vector with mean vector $\boldsymbol{\mu}\in \mathbb{R}^n$ and covariance matrix $\Sigma=(\sigma_{ij})_{n\times n}$.
\begin{enumerate}[(a)]
\item   The followings are equivalent:
(i) $\mathbf X$ is NA; (ii) $\mathbf X$ is NSD;  (iii) $\mathbf X$ is NUOD;    (iv) $\mathbf X$ is NLOD;
(v) $\mathbf X$ is NCD.
\item    The followings are equivalent:
(i) $\mathbf X$ is  JM; (ii) $\mathbf 1_n^\top \Sigma \mathbf 1_n=0$.
\item For $n=2$,  the followings are equivalent:
(i) $\mathbf X$ is JM; (ii) $\mathbf X$ is CT and $\sigma_{11}=-\sigma_{12}$.
\item For $n\ge 3$,   $\mathbf X$ is never CT unless   at least $n-2$ components of $\mathbf X$ are degenerate.
\end{enumerate}
\end{proposition}

Proposition \ref{prop:r1-G-1} shows the convenient property of the Gaussian distribution that the concepts of NA, NSD, NOD, NLOD, NUOD and NCD are all equivalent for this class. 
Parts (a) and (b) immediately tell that, for $n\ge 3$,
JM does not imply any of these concepts,
and none of these concepts implies JM. 
We will focus mostly on NA, NOD and NCD given their popularity and relative strength in the chains \eqref{eq:chain-2} and \eqref{eq:chain}.
For some other notions of negative dependence, see \cite{J14}.


\section{JM and negative dependence}
\label{sec:JM}

In this section, we explore the relation between JM and negative dependence concepts introduced in Section~\ref{sec:notions:negative:dependence} by means of several theoretical results. 
We first show  that a joint mix is NA under some properties of conditional independence and monotonicity.

\begin{theorem}\label{th:r1-1}
Let $\mathbf X $ be a joint mix and write $S_A =\sum_{i\in A} X_i$ for $A\subseteq [n]$. Suppose that
\begin{enumerate}[(a)]
\item 
 $\mathbf X_A$ and $\mathbf X_{[n]\setminus A}$ are  independent conditionally on $S_A  $ for every $A\subseteq [n]$;
 \item $\E\left[f(\mathbf X_A) |  S_A \right] $ is increasing in $S_A$ for every increasing function $f$  and $A\subseteq [n]$.\end{enumerate}  
Then $\mathbf X$ is NA.
\end{theorem}

Theorem \ref{th:r1-1} can be compared with Theorem 2.6 of \cite{JP83}, which says that if $\mathbf X$ is independent and satisfies (b), then 
the conditional distribution of $\mathbf X$ given $S_{[n]} $ is NA.
Since $S_{[n]} $ is a constant for JM and 
(a) is implied by independence,
Theorem \ref{th:r1-1} means that the NA condition in Theorem 2.6 of \cite{JP83} holds if
the independence assumption  is weakened to  conditional independence (a), and in addition  we assume JM. 
Note, however, that JM and independence of $\mathbf X$ conflict each other unless $\mathbf X$ is degenerate. 

Most existing examples of NA random vectors are presented by \cite{JP83}.
Although Theorem~\ref{th:r1-1} 
does not directly give new examples of NA random vectors, it can be used to check NA in popular examples. 

\begin{example}
We use Theorem~\ref{th:r1-1} to check that the uniform distribution on the standard simplex $\Delta_{n}=\{(x_1,\dots,x_n)\in [0,1]^n: \sum_{i=1}^n x_i=1\}$ is NA.  
Let $\mathbf X$ follow the uniform distribution over $\Delta_n$ which is JM.
For every $A \subseteq [n]$, we can check that $(\mathbf{X}_A,\mathbf{X}_{[n]\setminus A}) | \{S_A=s\}$  for $s\in (0,1)$ follows a uniform distribution on $ (s\Delta_n)\times ((1-s)\Delta_n)$ and condition (a) holds.
Condition (b) follows by noting that $\mathbf{X}_A|\{S_A=s\} \laweq s \mathbf{X}_A  |\{S_A=1\}$ for $s\in (0,1)$ and thus  $\mathbf{X}_A|\{S_A=s\}$ is stochastically increasing in $s$. 
\end{example}

\begin{example}
The multinomial distribution is known to be NA \citep{JP83}.
We show this by virtue of Theorem~\ref{th:r1-1}. 
Let $\mathbf X\sim \operatorname{MN}_n(k,\mathbf{p})$ follow a multinomial distribution with $k$ trails, $n$ mutually exclusive events and event probabilities $\mathbf p=(p_1, \dots, p_n)$.
For $s\in \{0, \dots,k \}$ and every $A \subseteq [n]$ with $B=[n]\setminus A$, it holds that  $\mathbf{X}_A|\{S_A=s\}\sim \operatorname{MN}_{\vert A\vert}(s,\mathbf p_{A}/\sum_{i\in A}p_i)$ 
and $\mathbf{X}_{B}|\{S_A=s\}\sim 
\operatorname{MN}_{\vert B\vert}(k-s,\mathbf p_{B}/\sum_{i\in B}p_i)$, where $p_A=(p_i)_{i\in A}$ and $p_B=(p_i)_{i\in B}$.
Then conditions~(a) and (b) can be checked directly by calculation.
\end{example}

 We next focus on   exchangeable joint mixes, which exhibit some specific forms of negative dependence.  A random vector $\mathbf X=(X_1,\dots,X_n)$ (or its distribution) is called \emph{exchangeable} if $\mathbf X \laweq \mathbf X^{\pi}$ for all $\pi\in {\mathfrak S_n}$, where ${\mathfrak S_n}$ is the set of all permutations on $[n]$ and $\mathbf X^{\pi} = (X_{\pi(1)},\dots,X_{\pi(n)})$.
First, we note that if $\mathbf X$ is CT with identical marginal distributions equal to $F$, then the distribution of $\mathbf X$ is explicitly given by
$\p(\mathbf X\le \mathbf x) = (F(x_1)+\dots+F(x_n)-n+1)_+ $
for $\mathbf x=(x_1,\dots,x_n)\in \R^n$; see, for example,  Theorem~3.3 of \cite{PW15}.  Clearly, this distribution is exchangeable.
Moreover, 
for any given marginal distribution  $ F$ which is $n$-completely mixable,
there exists an exchangeable joint mix with marginals $F$; see~Proposition 2.1 of \cite{PRWW19}. Note that an exchangeable joint mix is NCD because each bivariate correlation coefficient is equal to $-1/(n-1)$.
The next proposition states that such an exchangeable joint mix is also negatively dependent in the sense of NSD, NUOD and NULD if so is $\mathbf X$.

\begin{proposition}\label{prop:non:exchangeable:to:exchangeable:nod}
If a univariate distribution function $F$ supports an NSD $n$-joint mix, then  $F$ supports an exchangeable NSD $n$-joint mix.
The statement holds true if NSD is replaced by  NOD, NUOD or NLOD.
\end{proposition}

One may wonder whether Proposition \ref{prop:non:exchangeable:to:exchangeable:nod} holds with NSD replaced by NA. 
Unfortunately, this question remains open, as our proof for Proposition \ref{prop:non:exchangeable:to:exchangeable:nod} does not extend to  NA.




Next, we present a necessary condition  for a tuple of distributions to support any negatively dependent joint mixes. 
This condition is also sufficient when the marginal distributions are Gaussian.
 
\begin{proposition}\label{prop:necessary}
If the tuple of distributions $(F_1,\dots,F_n)$ with finite variance vector $(\sigma_1^2,\dots,\sigma_n^2)$  supports an NCD joint mix,
then   \begin{equation}\label{eq:maxcond} 2 \max_{i\in [n]}\sigma_i^2 \le \sum_{i\in [n]}\sigma_i^2.
\end{equation}
\end{proposition}

  Since NCD is weaker than NOD and NA,
  the necessary condition \eqref{eq:maxcond} is also necessary for NOD and NA joint mixes.

 For a given tuple of distributions $(F_1,\dots,F_n)$ with   finite variance vector $(\sigma_1^2,\dots,\sigma_n^2) $,
 the condition \eqref{eq:maxcond} is not necessary for the existence of an NCD random vector, since an independent random vector supported by $(F_1,\dots,F_n)$ always exists and it is NCD.
More interestingly, the condition  \eqref{eq:maxcond}  is not necessary for the existence of a joint mix either. Indeed, as shown by \citet[Corollary 2.2]{WW16},
a useful necessary condition for a joint mix supported by $(F_1,\dots,F_n)$ to exist is
 \begin{equation}\label{eq:maxcond2} 2 \max_{i\in [n]}\sigma_i  \le \sum_{i\in [n]}\sigma_i .
\end{equation}
We note that \eqref{eq:maxcond} is strictly stronger than  \eqref{eq:maxcond2},
because for any $j\in [n]$, \eqref{eq:maxcond} gives
$$
\sigma_j^2 \le\sum_{i\in [n]\setminus \{j\}} \sigma_i^2 ~\Longrightarrow~
\sigma_j  \le
\left({ \sum_{i\in [n]\setminus \{j\}} \sigma_i^2} \right)^{1/2} \le \sum_{i\in [n]\setminus \{j\}} \sigma_i,
$$
which implies \eqref{eq:maxcond2}. It is clear that \eqref{eq:maxcond} and \eqref{eq:maxcond2} are not equivalent; for example, $(\sigma_1,\sigma_2,\sigma_3)=(2,2,3)$ satisfies \eqref{eq:maxcond2} but not \eqref{eq:maxcond}.

By Proposition 2.4 of \cite{WPY13}, if the marginal distributions $F_1,\dots,F_n$ are Gaussian, then
the condition \eqref{eq:maxcond2} is necessary and sufficient for a joint mix supported by $(F_1,\dots,F_n)$ to exist. Hence, the condition \eqref{eq:maxcond}, which is strictly stronger than \eqref{eq:maxcond2}, is not necessary for a joint mix to exist.
On the other hand,
\eqref{eq:maxcond} is generally not sufficient for an NCD joint mix to exist either, since it is well known that a joint mix may not exist even if the marginal distributions are identical.
Nevertheless, it turns out  that \eqref{eq:maxcond} is necessary and sufficient for an NCD or NA joint mix to exist for Gaussian marginals. 

\begin{theorem}
\label{th:nodfull}
A tuple of univariate Gaussian distributions
with variance vector $(\sigma_1^2,\dots,\sigma_n^2)$  supports an NCD or NA joint mix if and only if \eqref{eq:maxcond} holds, that is,  $2 \max_{i\in [n]}\sigma_i^2 \le \sum_{i\in [n]}\sigma_i^2$.
Moreover, such an NCD or NA joint mix can be chosen as a Gaussian random vector.
\end{theorem}

 The negative dependence concepts of NA,  NSD, NOD, NLOD, NUOD and NCD are all equivalent for multivariate Gaussian distributions, as we see in Proposition \ref{prop:r1-G-1}. Hence, \eqref{eq:maxcond} is also necessary and sufficient for an NSD, NOD, NLOD, or NUOD  joint mix to exist for Gaussian marginals.
For elliptical distributions (details in Section \ref{sec:elliptical}),
the corresponding statement to Theorem \ref{th:nodfull} holds for NCD but not the other forms of negative dependence; see Proposition \ref{prop:R1-el}.

\begin{example}
\label{ex:nodgaussian}
In case $n=3$,  for any marginals with variance vector $(\sigma_1^2,\sigma_2^2,\sigma_3^2)$, the covariance of a joint mix $\mathbf X$ is uniquely given by
\begin{equation*}
\Sigma =\left(
           \begin{array}{ccc}
             \sigma_1^2 & \frac{1}{2}(\sigma_3^2-\sigma_1^2-\sigma_2^2) & \frac{1}{2}(\sigma_2^2-\sigma_1^2-\sigma_3^2) \\
              \frac{1}{2}(\sigma_3^2-\sigma_1^2-\sigma_2^2) & \sigma_2^2 & \frac{1}{2}(\sigma_1^2-\sigma_2^2-\sigma_3^2) \\
             \frac{1}{2}(\sigma_2^2-\sigma_1^2-\sigma_3^2) & \frac{1}{2}(\sigma_1^2-\sigma_2^2-\sigma_3^2)  & \sigma_3^2 \\
           \end{array}
         \right);
\end{equation*}
see \citet[Corollary 6]{XY20} for this statement.
If $\mathbf X$ is Gaussian, it is clear that $\mathbf X$ is NA if and only if \eqref{eq:maxcond} holds.
In case $n\ge 4$, for Gaussian marginals
 we can obtain an explicit covariance matrix of an NA joint mix from the proof of Theorem \ref{th:nodfull}.
\end{example}

%


We end this section with   two decomposition results of a joint mix into NA joint mixes, one through a random vector decomposition, and one through a mixture decomposition.
We first establish a new result showing that any finitely supported discrete joint mix can be decomposed into a linear combination of  binary multinomial random vectors.
 A binary multinomial random vector   is a random vector $(X_1,\dots,X_n)$ taking values in $\{0,1\}^n$ such that $\sum_{i=1}^n X_i=1$; that is, exactly one of $X_1,\dots,X_n$ takes the value $1$.
By definition, binary multinomial random vectors are CT (hence NA)  and JM.

\begin{theorem}\label{th:r1-2}
Suppose that the random vector $\mathbf X$ takes values in a finite set.
Then, $\mathbf X$ is JM if and only if 
it can be represented as a finite linear combination of binary multinomial random vectors.
\end{theorem}

Theorem \ref{th:r1-2} generalizes Theorem 2 of \cite{W15} which has a  decomposition of   a joint mix  taking  nonnegative integer values as the sum of binary  multinomial random vectors. The assumption of finite support in Theorem \ref{th:r1-2} does not seem to be dispensable with the current proof techniques. 

Next, using the fact that the distribution of a joint mix can be written as a mixture of discrete uniform (DU) distributions on $n$ points in $\R^n$, we obtain the following decomposition.

\begin{proposition}\label{prop:mixture}
The distribution of any exchangeable joint mix with center $\mu$ can be written as a  mixture of distributions of exchangeable NA joint mixes with center $\mu$.
\end{proposition}


\section{A multi-marginal optimal transport problem}
\label{sec:otp}

To connect JM and negative dependence, a natural question is  whether  in some applications  negatively dependent joint mixes have   additional attractive properties that are not shared by other joint mixes. We present an optimal transport problem with uncertainty in this section where a combination of negative dependence and JM naturally appears.

In the multi-marginal optimal transport theory~\citep{S15, P15}, a general objective is 
$$
\mbox{to minimize }
\E[c(X_1,\dots,X_n)] \mbox{~~~subject to $X_i\sim F_i$, $i\in[n]$},
$$
where $c:\R^n\to\R$ is a cost function and $F_1,\dots,F_n$ are specified marginal distributions.
This problem is referred to as the \emph{Monge-Kantorovich problem}.
 In the context of this paper, the distributions $F_1,\dots,F_n$  are  on $\R$. 
 In all optimization problems we discussed in this section, the constraint is always   $X_i\sim F_i$ for each $i\in[n]$ with $F_1,\dots,F_n$ given, and we   assume that $F_1,\dots,F_n$ have finite second moments throughout this section.


\subsection{Optimal transport under uncertainty on the set of components}

We will consider a special class of cost functions, leading to the Monge-Kantorovich problem
\begin{equation}\label{eq:f}
\mbox{to minimize }
\E\left[f\left(\sum_{i=1}^n X_i\right)\right] \mbox{~~~subject to $X_i\sim F_i$, $i\in[n]$},
\end{equation} 
where $f:\R\to\R$ is a convex function. 
This special setting is important to JM because, 
assuming that a joint mix with marginal distributions $F_1,\dots,F_n$ exists, then any joint mix is an optimizer of  \eqref{eq:f} due to Jensen's inequality, and conversely, any optimizer of \eqref{eq:f} has to be a joint mix if $f$ is strictly convex. 
As discussed by \cite{PW15} and \cite{WW16}, one of the main motivations of JM is to solve optimization problems similar to \eqref{eq:f}. 

Since joint mixes  with given marginal distributions are not unique, 
we wonder whether a negatively dependent joint mix plays a special role 
among optimizers to \eqref{eq:f}. This is our main question to address.
   
Although each joint mix minimizes  \eqref{eq:f}, their distributions can be quite different.
For a concrete example,  suppose that the marginal distributions are standard Gaussian, and let $n$ be even.
With these marginals, $\mathbf X^{\rm E}\sim \mathrm N_{n}(\mathbf 0_n, P_n^{*})$ is an NA joint mix, where 
 $P_n^{*}$  is a matrix with diagonal entries being $1$ and off-diagonal entries being $-1/(n-1)$, and
$\mathbf X^{\rm A} =((-1)^i Z)_{i\in [n]}$, $Z\sim \mathrm N_1(0,1)$,
is another joint mix which is not NA.
Here,   ``E"  stands for ``exchangeable" and
``A" stands for ``alternating".
These two joint mixes have the same value $f(0)$  for  \eqref{eq:f}. Nevertheless, 
$\mathbf X^{\rm A}$ may be seen as undesirable in some situations, because some subgroups of its components are comonotonic.
Inspired by this,
we consider the cost of a subset $K\subseteq [n]$ of risks 
$
f\left(\sum_{i\in K} X_i\right)  .
$
If $K$ is known to the decision maker, then we are back to \eqref{eq:f} with $(X_i)_{i\in [n ]}$ replaced by $(X_i)_{i\in K}$.

In different applications, allowing  a flexible choice of $K$ may represent the absence of some risks in a risk aggregation pool,
 missing particles in a quantum system, an unspecified number of simulation sizes in a sampling program, or
 uncertainty on the  participation of some agents in a risk-sharing game.
In each context above, a decision maker may not know $K$ a priori, and hence she may be interested in minimizing a  weighted average of the cost, that is
\begin{equation}\label{eq:mu}
C_{\mu}^f(X_1,\dots,X_n):=\sum_{K\subseteq [n]} \E\left[f\left(\sum_{i\in K} X_i\right)\right] \mu(K),\end{equation}
where $\mu$ is a probability on the sample space  $2^{[n]}$, the power set of $[n]$, and $ \sum_{i\in K} X_i $ is set to $0$ if $K$ is empty;
here we slightly abuse the notation by setting $\mu(K)=\mu(\{K\})$, which should not lead to any confusion. 

We consider the formulation of uncertainty as in the framework of \cite{GS89}. 
With a probability on $2^{[n]}$ uncertain, we consider a set $\mathcal{M}$ of  probabilities   on $2^{[n]}$, called an uncertainty set. The formulation of \eqref{eq:mu} with uncertainty set $\mathcal M$ is  
\begin{equation}\label{eq:mu_uncertainty1}
\mbox{to minimize } 
\sup_{\mu \in \mathcal M} C_{\mu}^f(X_1,\dots,X_n)
\mbox{~~~subject to $X_i\sim F_i$, $i\in[n]$}.
\end{equation} 
The supremum represents a worst-case attitude towards uncertainty, which is axiomatized by \cite{GS89} in decision theory. 
We explain two simple special cases of \eqref{eq:mu_uncertainty1}.
First, by taking $\mathcal M$ as the set of all probabilities on $2^{[n]}$,
the objective in \eqref{eq:mu_uncertainty1} becomes
\begin{equation}\label{eq:uncertainty_f}
\sup_{\mu \in \mathcal M} C_{\mu}^f(X_1,\dots,X_n) = \max_{K\subseteq [n]}  \E\left[f\left(\sum_{i\in K} X_i\right)\right],
 \end{equation}
which represents the situation of having no information on $K$.
Second, by taking $\mathcal M$ as the set of all probabilities on $2^{[n]}$ supported by
sets $K$ of cardinality $|K|=k\in [n]$, the objective in \eqref{eq:mu_uncertainty1} becomes
\begin{equation}\label{eq:uncertainty2_f}
\sup_{\mu \in \mathcal M} C_{\mu}^f(X_1,\dots,X_n) = \max_{K\subseteq [n],~|K|=k}  \E\left[ f\left(\sum_{i\in K} X_i\right)\right], 
\end{equation} 
which represents the situation where  one knows  how large the subset $K$ is, but not precisely how it is composed.

The problem \eqref{eq:mu_uncertainty1} is generally difficult to solve. 
We will first focus on the homogeneous case where $F=F_1=\dots=F_n$, and this will be relaxed in Section \ref{sec:heter}. 
With this interpretation, it is natural to consider uncertainty sets $\mathcal M$ that are symmetric. 
We say  that $\mathcal M$ is \emph{symmetric} if $\mu \in \mathcal M$ implies $\mu_\pi \in \mathcal M$ for $\pi\in \mathfrak S_n$, where $\mu_\pi$ is a permutation of $\mu$, defined by $\mu_{\pi}(K)=\mu(\{\pi(i): i\in K\})$  for $K\subseteq [n]$. 

Recall that we are interested in whether a negatively dependent joint mix plays a special role 
among other joint mixes. The next proposition provides a   step in this direction. Also recall that $P_n^*$ is the correlation matrix with off-diagonal entries equal to $-1/(n-1)$.
\begin{proposition}\label{prop:r1-1} Suppose that $\mathcal M$ is symmetric, and $\mathbf X$ is a joint mix with identical marginals  $F$. 
Then there exists  an exchangeable NCD joint mix $\mathbf X^{\rm E}$ with marginals $F$ and  correlation matrix $P_n^{*}$ such that 
$$
\sup_{\mu \in \mathcal M}  C_{\mu}^f(\mathbf X^{\rm E}) 
\le 
\sup_{\mu \in \mathcal M}  C_{\mu}^f(\mathbf X)
$$
for all measurable functions $f:\R \to \R$.
\end{proposition} 

Proposition \ref{prop:r1-1} illustrates the intuition that among all joint mixes, the NCD ones with correlation matrix $P_n^{*}$  are better choices under uncertainty.
However, this does not answer whether such NCD joint mixes are optimizers to  our main optimal transport problem \eqref{eq:mu_uncertainty1}.
In the next section, we consider the quadratic cost, and show that indeed those NCD joint mixes are solutions to \eqref{eq:mu_uncertainty1} for the quadratic cost.




\subsection{Quadratic cost}

We consider the quadratic cost given by $f(x)=x^2$.  In this case, we denote by 
\begin{equation*}
C_{\mu}^2(X_1,\dots,X_n):=\sum_{K\subseteq [n]} \E\left[\left(\sum_{i\in K} X_i\right)^2\right] \mu(K),\end{equation*}
and \eqref{eq:mu_uncertainty1} becomes, assuming homogeneous marginals,
\begin{equation}\label{eq:mu_uncertainty_quad}
\mbox{to minimize }  \sup_{\mu \in \mathcal M}\sum_{K\subseteq [n]} \E\left[\left(\sum_{i\in K} X_i\right)^2\right]\mu(K),
\mbox{~~~subject to $X_i\sim F$, $i\in[n]$}.
\end{equation} 
Two other formulations related to the quadratic cost, the repulsive harmonic cost problem and the variance minimization problem, are discussed below in Examples \ref{rem:var} and \ref{rem:haromonic:cost}.

 It is clear that, for $\mu \in \mathcal M$, its permutation $\mu_{\pi}$ satisfies 
$
  C_{\mu}^2(\mathbf X) 
=
 C_{\mu_\pi}^2(\mathbf X).
$
    We first show that  the exchangeable NCD joint mix is a minimizer to \eqref{eq:mu_uncertainty_quad} if $\mathcal M$ is symmetric.
\begin{theorem}
    \label{prop:average:var}
Suppose that $F$ is $n$-completely mixable with finite variance and $\mathcal M$ is symmetric. Then, each NCD joint mix  with marginals $F$ and correlation matrix $P_n^{*}$   minimizes \eqref{eq:mu_uncertainty_quad}.
\end{theorem}

 If the marginal distribution $F$ is Gaussian, then we can replace NCD by NA, NSD or NOD in Theorem \ref{prop:average:var}, since for the Gaussian class NCD is equivalent to these notions (Proposition \ref{prop:r1-G-1}).


Theorem \ref{prop:average:var}  does not state that all the minimizers to \eqref{eq:mu_uncertainty_quad} are   NCD joint mixes.
For instance, if $\mathcal M$ contains only measures concentrated on $K$ with $|K|=n$, then any joint mix minimizes \eqref{eq:mu_uncertainty_quad}; see also Remark \ref{rem:kn} below for other similar cases. 
Next, we study the uniqueness of the optimizers for two special choices of $\mathcal M$ in \eqref{eq:uncertainty_f} and \eqref{eq:uncertainty2_f}, namely, 
\begin{equation}\label{eq:uncertainty_quad}
\mbox{to minimize }  \max_{K\subseteq [n]}  \E\left[\left(\sum_{i\in K} X_i\right)^2\right] \mbox{~~~subject to $X_i\sim F$, $i\in[n]$},
\end{equation}
 and for a fixed $k\in [n]$,
\begin{equation}\label{eq:uncertainty2_quad}
\mbox{to minimize } \max_{K\subseteq [n],~|K|=k}   \E\left[\left(\sum_{i\in K} X_i\right)^2\right] \mbox{~~~subject to $X_i\sim F$, $i\in[n]$}.
\end{equation}


 In Theorem \ref{th:opt} below, we will see that, assuming $F$ has mean zero, negative dependence yields more stable costs in the presence of  uncertainty. The exchangeable joint mix with correlation matrix $P^*_n$ minimizes \eqref{eq:uncertainty2_quad}
for each $k\in [n]$, and this correlation matrix is unique for all minimizers for each $k\in [n]\setminus\{1,n-1,n\}$.
As a consequence, all minimizers to \eqref{eq:uncertainty_quad} have the same correlation matrix $P_n^*$ (this holds for $n\ge 3$).



 
\begin{theorem}\label{th:opt}
Suppose that $n\ge 3$ and the distribution $F$ is $n$-completely mixable with mean $0$ and finite positive variance. 
A random vector is a minimizer to 
\eqref{eq:uncertainty_quad}
if and only if 
it is an NCD joint mix with correlation matrix $P_n^{*}$.
The same conclusion holds true  if \eqref{eq:uncertainty_quad} is replaced by
  \eqref{eq:uncertainty2_quad} with any $k\in [n]\setminus\{1,n-1,n\} $.
\end{theorem}

\begin{remark}\label{rem:kn}
We briefly comment on the three cases of $k$ excluded from the statement regarding the unique minimizer of \eqref{eq:uncertainty2_quad}, and it will be clear that uniqueness cannot be expected in these cases. Recall that  the marginal distributions of $\mathbf X$ are assumed identical.
\begin{enumerate}
\item If $k=1$,  then $\E [ (\sum_{i\in K} X_i )^2 ]=\E [X_1^2 ]$
which does not depend on the dependence structure of $\mathbf X$, and hence any coupling minimizes \eqref{eq:uncertainty2_quad}.
\item If $k=n$, then  $K=[n]$ and thus any joint mix  minimizes \eqref{eq:uncertainty2_quad}.
\item If $k=n-1$, then  $\E [(\sum_{i\in K} X_i )^2 ]=\E [(c-X_1)^2 ] $  
for any joint mix $\mathbf X$ with center $c$.
Hence, any joint mix
has the same  value for  \eqref{eq:uncertainty2_quad}.
 \end{enumerate}
 \end{remark}

Theorem \ref{th:opt} implies that
for a standard Gaussian $F$, the exchangeable joint mix $\mathbf X^{\rm E}\sim \mathrm N_{n}(\mathbf 0_n, P_n^{*})$ is a minimizer to both \eqref{eq:uncertainty_quad} and \eqref{eq:uncertainty2_quad} for each $k\in [n]$.
 If $n\ge 3$, this minimizer is unique among Gaussian vectors in both cases of \eqref{eq:uncertainty_quad}  and  \eqref{eq:uncertainty2_quad} with $k\in [n]\setminus\{1,n-1,n\} $.
 \begin{remark}
 As we have seen in Remark \ref{rem:kn}, if  $n=3$, then
any joint mix minimizes \eqref{eq:uncertainty2_quad} for each $k\in [n]$.
The uniqueness statement in  Theorem \ref{th:opt}
implies that the covariance structure of a joint mix is unique for $n=3$, as we see in Example \ref{ex:nodgaussian}.
\end{remark}

Below we discuss two specific optimal transport problems related to the quadratic cost. 
\begin{example}[Variance minimization]\label{rem:var}
The quadratic cost minimization  problem  is  equivalent to variance minimization with given marginals.  It is clear that 
$$C^2_\mu(X_1, \dots, X_n)=\sum_{K\subseteq [n]} \var\left(\sum_{i\in K} X_i\right)\mu(K)+ \sum_{K\subseteq [n]} \left(\sum_{i\in K}\E[ X_i]\right)^2\mu(K),$$
and the second term does not depend on the dependence structure of $(X_1, \dots, X_n)$. If $F_1,\dots,F_n$ have zero mean, then the problem \eqref{eq:mu_uncertainty_quad} can be written as
\begin{equation*}
\mbox{to minimize } \sup_{\mu \in \mathcal M}\sum_{K\subseteq [n]} \var\left(\sum_{i\in K} X_i\right)\mu(K)\mbox{~~~subject to $X_i\sim F$, $i\in[n]$}.
\end{equation*}  
Variance minimization is a classic problem in Monte Carlo simulation~\citep{CM01,CM05} and risk management~\citep{R13}.  
The above arguments show that the statements in Theorems \ref{prop:average:var} and  \ref{th:opt} hold true if 
the objective of quadratic cost $\E[(\sum_{i\in K} X_i)^2]$ is replaced by the variance   $\var(\sum_{i\in K} X_i)$. 

\end{example}

\begin{example}[Repulsive harmonic cost]\label{rem:haromonic:cost}
The \emph{repulsive harmonic cost} function is defined by
$$
c(x_1,\dots,x_n)= -\sum_{i,j=1}^n (x_i-x_j)^2,~~~~~(x_1,\dots,x_n)\in \R^n.
$$
This cost function originates from the so-called weak interaction regime in Quantum Mechanics; see e.g., \cite{DGN17}.
Any joint mix minimizes the expected repulsive harmonic cost. To see this, we can rewrite
\begin{align}\label{eq:variance}
\E[c(X_1,\dots,X_n)] &= -2n\sum_{i=1}^n \E[X_i^2] + 2 \E\left[\left(\sum_{i=1}^n X_i\right)^2\right]  
.
\end{align}
Since the first  terms on the right-hand side of \eqref{eq:variance} do not depend on the dependence structure of $(X_1,\dots,X_n)$, minimizing $ \E[c(X_1,\dots,X_n)] $ is equivalent to minimizing $\E [ (\sum_{i=1}^n X_i )^2 ] $,  
which is clearly minimized if $(X_1,\dots,X_n)$ is a joint mix. 
Let $c_K(x_1,\dots,x_n)=-\sum_{i,j\in K}(x_i-x_j)^2$, $(x_1,\dots,x_n)\in \R^n$, for $K\subseteq [n]$.
The  problem \eqref{eq:mu_uncertainty_quad} can be written as
\begin{equation*}
\mbox{to minimize } \sup_{\mu \in \mathcal M}\sum_{K\subseteq [n]}\left( \frac 12 \E[c_K(\mathbf X)] + |K| \sum_{i\in K}\E[X_i^2] \right) \mu(K)\mbox{~~~subject to $X_i\sim F$, $i\in[n]$}.
\end{equation*} 
The statement in Theorem \ref{prop:average:var}  remains true if 
the objective of quadratic cost $\E[(\sum_{i\in K} X_i)^2]$ is replaced by the cost  $\E[c_K(\mathbf X)]$.  
\end{example}



\subsection{Discussions on heterogeneous marginals}
\label{sec:heter}

In  Theorem \ref{th:opt}, we assumed that the marginal distributions are identical. This assumption is not  dispensable, as the situation for heterogeneous marginals is drastically different and we do not have general results.  
In this section, we present a result in the simple case $n=3$ and provide several examples to discuss some subtle issues and open questions. 
To illustrate these issues, 
we focus on the problems
\eqref{eq:uncertainty_quad} and \eqref{eq:uncertainty2_quad} for $n=3$.
In all examples, we explain with Gaussian marginal distributions, but this assumption can be replaced as long as the covariance matrices in the examples are compatible with the marginals.

\begin{proposition}\label{prop:opt-n3}
Let $n=3$.
For any tuple of marginal distributions with finite variance vector and zero means, any joint mix, if it exists, minimizes \eqref{eq:uncertainty_quad}.  If an NCD joint mix exists, then no random vector with any positive bivariate covariance can minimize  \eqref{eq:uncertainty_quad}.
\end{proposition}

\begin{remark}\label{rem:conjecture}
Proposition \ref{prop:opt-n3} states that, if a Gaussian triplet supports an NCD joint mix, then it minimizes \eqref{eq:uncertainty_quad}, and all Gaussian minimizers must be NCD.
It is not clear whether this observation can be extended to $n\ge 4$.
\end{remark}

Unlike the situation in Theorem \ref{th:opt}, uniqueness of the covariance matrix does not hold in the setting of Proposition \ref{prop:opt-n3},
as illustrated in the following example.
\begin{example}
Consider  Gaussian marginal distributions  with variance vector $(\sigma_1^2,\sigma_2^2,\sigma_3^2)=(2,1,1)$. In this case, \eqref{eq:maxcond}  holds, and an NCD joint mix exists by Theorem~\ref{th:nodfull}.
Both the covariance matrices $\Sigma$ and $\Sigma'$ defined by
 $$
\Sigma =\left(
           \begin{array}{ccc}
             2 &  -1 & -1 \\
              -1 & 1 & 0 \\
              -1 & 0 & 1\\
           \end{array}
         \right)
 \mbox{~~~and~~~}
\Sigma' =\left(
           \begin{array}{ccc}
             2 &  -1/2 & -1 \\
              -1/2 & 1 & 0 \\
              -1 & 0 & 1\\
           \end{array}
         \right)
 $$
 minimize \eqref{eq:uncertainty_quad} subject to the marginal distributions.
We can see that $\Sigma$ corresponds to an NCD joint mix, whereas   $\Sigma'$ corresponds to an NCD  random vector, but not a joint mix.
\end{example}

The next example illustrates that, although a joint mix generally minimizes \eqref{eq:uncertainty_quad} in case $n=3$, NCD may be more relevant than joint mixes for minimizing \eqref{eq:uncertainty2_quad} with some   $k\ne n$ when the two dependence requirements cannot be simultaneously achieved.

\begin{example}
\label{ex:9}
Consider Gaussian marginal distributions  with variance vector $(\sigma_1^2,\sigma_2^2,\sigma_3^2)=(4,1,1)$. In this case, \eqref{eq:maxcond} does not hold, and no NCD joint mix exists. 
Both the covariance matrices $\Sigma$ and $\Sigma'$ defined by
 $$
\Sigma =\left(
           \begin{array}{ccc}
             4 &  -2 & -2 \\
              -2 & 1 & 1 \\
              -2 & 1 & 1\\
           \end{array}
         \right)
 \mbox{~~~and~~~}
\Sigma' =\left(
           \begin{array}{ccc}
            4 &  -1 & -1 \\
              -1 & 1 & 0 \\
              -1 & 0 & 1\\
           \end{array}
         \right)
 $$
  minimize \eqref{eq:uncertainty_quad} subject to the marginal distributions.
The covariance matrix $\Sigma$ corresponds to a joint mix, but not NCD. The covariance matrix $\Sigma'$ corresponds to an NCD  random vector, but not a joint mix.
Thus, the problem \eqref{eq:uncertainty_quad} admits an  NCD minimizing distribution $\mathrm N_3(\mathbf 0_3,\Sigma')$.
Moreover, for \eqref{eq:uncertainty2_quad} with $k=2$, the NCD distribution $\mathrm N_3(\mathbf 0_3,\Sigma')$ has a maximum of $3$ which
is strictly better than the joint mix distribution $\mathrm N_3(\mathbf 0_3,\Sigma)$ with a maximum of $4$.
\end{example}

Example~\ref{ex:9} suggests, informally, that there is a trade-off between a joint mix and NCD when both cannot be attained simultaneously, with a joint mix minimizing \eqref{eq:uncertainty2_quad} for $k=n$,
and an NCD random vector improving \eqref{eq:uncertainty2_quad} from the case of a joint mix for some $1<k<n$.
In fact, \eqref{eq:uncertainty2_quad} is not always minimized by NCD random vectors as seen in the following example.

\begin{example}\label{ex:not:nod:opt}
Consider   marginal distributions with variance vector $(\sigma_1^2,\sigma_2^2,\sigma_3^2)=(\sigma^2,1,1)$ and zero
means, where $\sigma>3$.
For any $(X_1,X_2,X_3)$ with the given marginals, we have
\begin{align*}
 \E\left[(X_2+X_3)^2\right]& \le 4 < (\sigma-1)^2   \leq \min\left(\E\left[(X_1+X_2)^2\right],\E\left[(X_1+X_3)^2\right]\right),
 \end{align*}
and hence 
\begin{equation}\label{eq:not:nod:opt}
\max_{K\subseteq[3],~|K|=2}\E\left[\left(\sum_{i\in K}X_i\right)^2\right] = \sigma^2+1+2\sigma \max(\rho_{12},\rho_{13}),
 \end{equation}
where $\rho_{ij}$, $i,j\in[3]$ is the correlation coefficient of $(X_i,X_j)$. 
For Gaussian marginals, 
the minimum of~\eqref{eq:not:nod:opt} is attained if and only if $\rho_{12}=\rho_{13}=-1$.
In this case, $\rho_{23}=1$ is the only possible correlation, and thus the minimizer to~\eqref{eq:not:nod:opt} cannot be NCD.
\end{example}

On the other hand, the next example shows that, if $n=3$, there always exists an NCD minimizer to \eqref{eq:uncertainty_quad} for Gaussian marginals.

\begin{example}~\label{ex:nod:min}
Let $(X_1,X_2,X_3)$ follow 
a multivariate Gaussian distribution with 
equicorrelation matrix $P_3^\ast$; i.e., all pairwise correlation coefficients are $-1/2$.
The variances $\sigma_1^2$,  $\sigma_2^2$ and  $\sigma_3^2$ are assumed to satisfy $\sigma_1\leq \sigma_2\leq \sigma_3$ without the loss of generality and the means for prescribed marginal distributions are assumed to be zero. We can easily verify that
each of $\E\left[(X_1+X_2+X_3)^2\right]$ and $\E\left[(X_i+X_j)^2\right]$, $i,j\in [3]$, is smaller than or equal to $\sigma_3^2$.
Hence, $(X_1,X_2,X_3)$ attains the lower bound
$$
\max_{K\subseteq[3]}\E\left[\left(\sum_{i\in K}X_i\right)^2\right] = \sigma_3^2 = \max_{i \in [3]}\sigma_i^2,
$$
and thus it minimizes \eqref{eq:uncertainty_quad}.
\end{example}

\begin{remark}\label{rem:conj-2}
For $n\geq 4$, it is not clear whether there always exists an NCD minimizer to \eqref{eq:uncertainty_quad} under a general heterogeneous  marginal constraint.
\end{remark}




\section{Elliptical distributions}\label{sec:elliptical}


Elliptical distributions form a tractable class of joint mixes for arbitrary dimensions.
In this section, we investigate negative dependence properties of such elliptical joint mixes.

An $n$-dimensional \emph{elliptical distribution} is a family of multivariate distributions
defined through the characteristic function
\begin{align}\label{eq:ellip}
\phi_{ \mathbf X}(\mathbf t)=\E\left[\exp\left({\mathrm i}\mathbf t^\top  \mathbf X \right)\right]=\exp\left({\mathrm i} \mathbf t^\top \boldsymbol \mu  \right)\psi(\mathbf t^\top \Sigma\mathbf t),~~~~\mathbf t\in \R^n,
\end{align}
for some location parameter $\boldsymbol{\mu}\in \mathbb R^n$, $n\times n$ positive semi-definite symmetric matrix $\Sigma\in \mathbb R^{n\times n}$ and the so-called \emph{characteristic generator} $\psi:\R_+\to \R$, where $\R_+=\{x\in \R: x\geq 0\}$. See Section~6 of~\cite{MFE15} for more properties.
We denote an elliptical distribution by $\operatorname{E}_n(\boldsymbol{\mu}, \Sigma,\psi)$ and refer to $\boldsymbol{\mu}$ as the \emph{location vector} and $\Sigma$ the \emph{dispersion matrix}.
We say that an elliptical distribution $\mathrm E_n(\boldsymbol \mu, \Sigma,\psi)$ is non-degenerate if all its marginals are non-degenerate (i.e., not a point-mass). Equivalently, the diagonal entries of $\Sigma$ are positive.
As presented in Proposition~6.27 of~\cite{MFE15}, a random vector $\mathbf X\sim \operatorname{E}_n(\boldsymbol{\mu}, \Sigma,\psi)$ with $\operatorname{rank}(\Sigma)=k$ admits the stochastic representation $\mathbf X = \boldsymbol{\mu} + RA\mathbf S$, where $\mathbf S$ is the uniform distribution on the unit sphere on $\mathbb R^k$, the radial random variable $R\geq 0$ is  independent of $\mathbf S$, and $A \in \mathbb R^{n\times k}$ is such that $A A^\top=\Sigma$.
With this representation, we have that $\mathbb E[\mathbf X]=\boldsymbol \mu$ and $\cov(\mathbf X)={\mathbb E}[R^2]\,\Sigma/k$ provided $\E[R^2]<\infty$.

We first present a simple lemma on elliptical joint mixes which will be useful for later discussions.\begin{lemma}\label{lem:simple}
An $n$-dimensional elliptically distributed random vector $\mathbf X \sim \operatorname{E}_n(\boldsymbol{\mu}, \Sigma,\psi)$ is a joint mix if and only if $\mathbf{1}_n^\top\Sigma\mathbf{1}_n=0$ or $\psi =1$ on $\R_+$.
\end{lemma}
Negative dependence of such an elliptical joint mix is the main topic of this section.
We next provide a characterization for NCD-JM as an extension to Theorem~\ref{th:nodfull}.

\begin{proposition}
\label{prop:R1-el} 
Suppose that  $\psi$ is the characteristic generator of an $n$-dimensional elliptical distribution.
A tuple of  univariate  distributions $(\mathrm{E}_1(\mu_i, \sigma_i^2, \psi),i\in [n])$ 
supports an NCD  joint mix if and only if \eqref{eq:maxcond} holds, that is,  $2 \max_{i\in [n]}\sigma_i^2 \le \sum_{i\in [n]}\sigma_i^2$.
Moreover, such an NCD joint mix can be chosen to follow an elliptical distribution.
\end{proposition}


As we showed in Proposition \ref{prop:r1-G-1}, for Gaussian random vectors, NA, NSD and NOD are all equivalent to NCD; that is, the  bivariate correlations are non-positive. 
 Recall that $P_n^{*}$ is an $n\times n$ matrix whose diagonal entries are $1$ and off-diagonal entries are $-1/(n-1)$.
Together with Lemma~\ref{lem:simple}, the matrix $P_n^{*}$ is the only choice of $\Sigma$ with diagonal entries being $1$ such that $\mathbf X\sim \mathrm{N}_n(\boldsymbol \mu,\Sigma)$ is an exchangeable NA (and thus NSD and NOD) joint mix.

One may hope that non-Gaussian elliptical distributions can represent NOD, NSD and NA joint mixes for $n\ge 3$.
The following result states that Gaussian family is characterized as the only   elliptical family which admits such a negatively dependent $n$-joint mix for all $n$.
For a characteristic generator $\psi$,
denote by $\mathcal {E} (\psi)$  the class of all non-degenerate random vectors following an elliptical distribution with characteristic generator $\psi$.
In what follows, a class $\mathcal E(\psi)$ is a \emph{Gaussian variance mixture family} 
if there exists a nonnegative random variable $W$  such that each member $\mathbf X$ admits the stochastic representation $\mathbf X \laweq  \boldsymbol{\mu} + \sqrt{W}A\mathbf Z$, where $\boldsymbol{\mu}\in \mathbb{R}^n$, $A\in \mathbb{R}^{n\times k}$, and $\mathbf{Z}$ is a $k$-dimensional standard Gaussian independent of $W$.



\begin{theorem}
\label{th:impossible}
%
%
Let $\psi$ be a characteristic generator.
\begin{enumerate}[(i)]
\item\label{item:ncd} The class $\mathcal {E} (\psi)$  contains an NCD $n$-joint mix for all $n\ge 2$
if and only if the class $\mathcal {E} (\psi)$ is a Gaussian variance mixture family.
\item\label{item:nod} The class $\mathcal {E} (\psi)$  contains an NOD, NSD, or NA $n$-joint mix for all $n\ge 2$
if and only if the class $\mathcal {E} (\psi)$ is Gaussian.
\end{enumerate}
\end{theorem}

Theorem~\ref{th:impossible} shows a clear contrast between NCD and other concepts of negative dependence.
As seen in the proof of Theorem~\ref{th:impossible}, NCD does not restrict the class of elliptical distributions since $\psi$ generates an $n$-dimensional elliptical distribution for every $n\in \N$ if and only if the corresponding elliptical class is a  Gaussian variance mixture family~\citep[Section~2.6]{FKN90}.
Note that a class of multivariate $t$ distributions (with a common degree of freedom) is an example of a Gaussian variance mixture family.
On the other hand, NOD, NSD and NA characterize Gaussian. This result stems from the fact that multivariate Gaussian distribution is the only one among elliptical distributions such that independence is equivalent to uncorrelatedness.

\section{Conclusion}\label{sec:conclusion}

The paper has focused on  the relationship between JM and   classic notions of negative dependence such as NOD and NA. 
Various connections between these concepts are obtained, and some  conditions for a joint mix to be negatively dependent are derived.
In particular, an exchangeable negatively dependent joint mix solves a multi-marginal optimal transport problem for quadratic cost under uncertainty on the participation of agents.



Negative dependence is always studied with many technical challenges. 
Although our main questions are addressed or partially addressed in this paper,  they give rise to many questions that remain open. We list a few of them that we find particularly interesting. 
\begin{enumerate}
\item Under what conditions, possibly stronger than exchangeability and NOD, is a joint mix  NA? 
An example of NOD joint mix that is not NA can be found in Section 3.2 of \cite{MR23} in the context of knockout tournaments with a nonrandom draw.
\item Under what general conditions, other than Gaussian, do we know a tuple of distributions supports an NA joint mix? For a fixed $n\ge 3$, this question is not clear even within the elliptical class.
\item It is unclear whether the decomposition result in Theorem \ref{th:r1-2} can be generalized to joint mixes that take infinitely many values or that are continuously distributed. 
\item Assuming homogeneous marginals, does an exchangeable joint mix solve problem \eqref{eq:mu_uncertainty1} for a general convex cost function $f$? 
In  Theorem \ref{prop:average:var}, we showed that this holds true for quadratic cost. 
We also know that a joint mix is an optimizer for the general convex cost problem without uncertainty. These observations seem to hint at the possible optimality of some exchangeable joint mix for general convex cost under uncertainty, but we do not have a proof. 
\item  Do negatively dependent joint mixes play an important role in optimization problems other than the ones considered in Section \ref{sec:otp}?
It is also unclear how 
 results in Section \ref{sec:otp} can be extended to heterogeneous marginal distributions with dimension higher than $3$. Two unsolved questions have already been mentioned in Remarks \ref{rem:conjecture} and \ref{rem:conj-2}.
\end{enumerate}
These questions yield new challenges to dependence theory and require future research.


\section*{Acknowledgements}
Takaaki Koike was supported by JSPS KAKENHI Grant Number JP21K13275. Ruodu Wang acknowledges financial support from the Natural Sciences and Engineering Research Council of Canada (RGPIN-2018-03823, RGPAS-2018-522590).



\appendix


\section{Proofs of all results}\label{sec:proofs}

Some notions of negative dependence introduced in Section~\ref{sec:notions:negative:dependence} are related to stochastic orders.
We first introduce some concepts of stochastic order. 

For two $n$-dimensional random vectors $\mathbf X$ and $\mathbf Y$, $\mathbf X$ is said to be less than $\mathbf Y$ in
 \emph{lower concordance order} (denoted by $\mathbf X\leq_{\rm cL}\mathbf Y$) if $\p(\mathbf X\leq \mathbf t)\leq \p(\mathbf Y \leq \mathbf t)$ for all $\mathbf t \in \R^n$,
 \emph{upper concordance order} (denoted by $\mathbf X\leq_{\rm cU}\mathbf Y$) if  $\p(\mathbf X > \mathbf t)\leq \p(\mathbf Y > \mathbf t)$ for all $\mathbf t\in \R^n$,
 \emph{concordance order} (denoted by $\mathbf X\leq_{\rm c}\mathbf Y$) if $\mathbf X\leq_{\rm cL}\mathbf Y$ and $\mathbf X\leq_{\rm cU}\mathbf Y$, and 
 in \emph{supermodular order} (denoted by $\mathbf X\leq_{\rm sm}\mathbf Y$) if $\mathbb E[\psi(\mathbf X)]\leq \mathbb E[\psi(\mathbf Y)]$ for all supermodular functions $\psi:\R^n \rightarrow \R$ such that the expectations exist.  
Using these notations of stochastic order, the notions of negative dependence NLOD, NUOD, NOD and NSD for an $n$-dimensional random vector $\mathbf{X}=(X_1,\dots,X_n)$ are denoted by $\mathbf{X}\le_{\rm cL}\mathbf{X}^\perp$, $\mathbf{X}\le_{\rm cU}\mathbf{X}^\perp$, $\mathbf{X}\le_{\rm c}\mathbf{X}^\perp$ and
$\mathbf X\leq_{\rm sm}\mathbf X^\perp$, respectively, where we recall that ${\mathbf X}^\perp=(X_1^\perp,\dots,X_n^\perp)$ is a random vector with independent components such that $X_i \laweq X_i^\perp$, $i\in [n]$.

\begin{proof}[Proof of Proposition~\ref{prop:r1-G-1}]
In part (a), the implication from (v) to (i) is shown by \cite{JP83}. 
The other implications follow from \eqref{eq:chain}.
Parts (b) and (c) can be easily checked by definition. 
Finally, part (d) follows from the fact that a CT random vector for $n\ge 3$ cannot have continuous marginal distributions~\citep{D72,PW15}.
\end{proof}

\begin{proof}[Proof of Theorem~\ref{th:r1-1}]
Note that to show NA, it suffices to show \eqref{eq:def-na} for $A,B$ that form a partition of $[n]$, as we can choose increasing functions in \eqref{eq:def-na} that only depend on a subset of $A$ and $B$. 
Let   $f$ and $g$ be two increasing functions on $\R^d$ and $\R^{n-d}$, respectively, where $d$ is the cardinality of $A$.
Note that
$$
\cov(f(\mathbf X_A),g(\mathbf X_B)) 
= \E[\cov ( f(\mathbf X_A)  , g(\mathbf X_B)|S_A )]  + \cov (\E[f(\mathbf X_A)|S_A ], \E[g(\mathbf X_B)|S_A ]);
$$
see (1.1) of \cite{JP83}.
Using conditional independence (a), we get 
$$
\cov(f(\mathbf X_A),g(\mathbf X_B)) 
=  \cov (\E[f(\mathbf X_A)|S_A ], \E[g(\mathbf X_B)|S_A ]).
$$
Since $S_A +S_B $ is a constant, condition (b) implies that 
$\E[f(\mathbf X_A)|S_A ]$ is an increasing function of $S_A$, and $\E[g(\mathbf X_B)|S_A ]=\E[g(\mathbf X_B)|S_B ]$ is a decreasing function of $S_A$. This shows that their covariance is non-positive. Therefore, $\mathbf X$ is NA.
\end{proof}

\begin{proof}[Proof of Proposition~\ref{prop:non:exchangeable:to:exchangeable:nod}]
Denote by $\mathbf X=(X_1,\dots,X_n)$ the NSD $n$-joint mix with joint distribution $F_{\mathbf X}$.
Let $\mathbf X^{\Pi}=(X_{\Pi(1)},\dots,X_{\Pi(n)})$ be an exchangeable joint mix, where $\Pi$ follows a uniform distribution on ${\mathfrak S_n}$ and is independent of $\mathbf X$.
Obviously $\mathbf X^{\Pi}$ is a joint mix, and has the same marginal distributions as $\mathbf X$.
Let $\bar F=\frac{1}{n!}\sum_{\pi \in {\mathfrak S_n}}F_{\mathbf X^\pi}$ be the distribution function of $\mathbf X^{\Pi}$.
Then $\bar F$ is exchangeable.
Moreover, $\bar F$ is NSD since
\begin{align*}
\E[\psi(\mathbf X^{\Pi})]=\frac{1}{n!}\sum_{\pi \in {\mathfrak S_n}}\E[\psi(\mathbf X^{\Pi})]\le
\frac{1}{n!}\sum_{\pi \in {\mathfrak S_n}}\E[\psi(\mathbf X^{\perp})]=\E[\psi(\mathbf X^{\perp})]
\end{align*}
for every supermodular function $\psi$ such that the expectations above exist.
Other cases of NOD, NUOD and NLOD are shown analogously.
\end{proof}

\begin{proof}[Proof of Proposition~\ref{prop:necessary}]
Without loss of generality, assume $\sigma_n^2$ is the maximum of $\{\sigma_1^2,\dots,\sigma_n^2\}$.
Note that NCD implies that the bivariate correlations are non-positive.
If $(X_1,\dots,X_n)$ is an NCD joint mix where $X_i\sim F_i$, $i\in [n]$, then
$$\sigma_n^2 = \var(X_n)= \var(X_1+\dots+X_{n-1}) \le  \sum_{i=1}^{n-1} \var(X_i) = \sum_{i=1}^{n-1} \sigma_i^2,$$
which yields \eqref{eq:maxcond} by adding $\sigma_n^2$ to both sides.
  \end{proof}

  \begin{proof}[Proof of Theorem~\ref{th:nodfull}]
The necessity follows from Proposition \ref{prop:necessary}, and below we show sufficiency.  Suppose that \eqref{eq:maxcond} holds.  Without loss of generality, we can assume $\sigma_n\ge \sigma_{n-1}\ge\cdots\ge\sigma_1$.
It suffices to consider $n\ge 3$ and $ \sigma_{n-1}>0$, and otherwise the problem is trivial. Moreover, the location parameters of the Gaussian distributions are not relevant, and they are assumed to be $0$. 

Let $\lambda$ be a constant such that
\begin{equation}\label{eq:th3lambda}
\lambda^2 \sum_{i=1}^{n-1} \sigma_i^2  + (1-\lambda^2)\sigma_{n-1}^2 = \sigma_n^2.
\end{equation}
By \eqref{eq:maxcond}, we have $   \sum_{i=1}^{n-1} \sigma_i^2 \ge  \sigma_n^2 \ge \sigma_{n-1}^2 $,
and this ensures that we can take $\lambda \in [0,1]$.

Let $P_n^{*}$  be a matrix with diagonal entries being $1$ and off-diagonal entries being $-1/(n-1)$,  and let $P_n^\perp$  be a matrix with diagonal entries being $1$ and off-diagonal entries being $0$. Take $\mathbf Y= (Y_1,\dots,Y_{n-1})\sim \mathrm{N}_{n-1}(\mathbf 0_{n-1}, P_{n-1}^\perp)$ and   $$\mathbf Z^{(m)}=(Z_m^{(m)},\dots,Z_{n}^{(m)})\sim \mathrm N_{n-m+1}(\mathbf 0_{n-m+1}, P_{n-m+1}^{*}), ~~~~~m =1,\dots,n-1,$$ such that $\mathbf Y,\mathbf Z^{(1)},\dots,\mathbf Z^{(n-1)}$ are independent.
 Note that $\mathbf Z^{(n-1)} =(Z_{n-1}^{(n-1)}, Z_{n}^{(n-1)}) \sim \mathrm N_{2}(\mathbf 0_{2}, P_{2}^{*})$ is $2$-dimensional,
 and each   $\mathbf Z^{(m)}$ is a joint mix.

 For notational simplicity, let the function $d $ be given by $d(a,b)=(a^2-b^2)^{1/2}$ for $a\ge b\ge 0.$
 Note that $a^2= d(a,b)^2 + b^2$.
 Moreover, for $k=1,\dots,n-1$, let $$\alpha_k =   d(\sigma_{k},\sigma_{k-1})  = \left(\sigma_{k}^2-\sigma_{k-1}^2\right)^{1/2},  $$
 with $\sigma_0=0$, and thus $\alpha_1=\sigma_1$.
For $k=1,\dots,n-1$, let
   $$
 X_k = \lambda \sigma_k Y_k+  d(1,\lambda)
 \sum_{j=1}^k  \alpha_j Z^{(j)}_k    .
 $$
Moreover, let
 $$
 X_n= -\lambda Y^*  +
 d(1,\lambda)
 \sum_{j=1}^{n-1}  \alpha_j Z^{(j)}_n ,\quad \mbox{where~}
 Y^* =   \sum_{k=1}^{n-1}   \sigma_k Y_k .
 $$
For $k=1,\dots,n-1$,
using independence among $Z^{(1)}_k,\dots,Z^{(k)}_k$, we get
 $$
 \var \left( \sum_{j=1}^k  \alpha_j Z^{(j)}_k\right)=
 \sum_{i=1}^k \alpha_j^2 =
 \sigma_1^2 + d(\sigma_2,\sigma_1)^2 + \dots + d(\sigma_{k},\sigma_{k-1})^2 = \sigma_k^2.
 $$
Hence,
 $ \sum_{j=1}^k  \alpha_j Z^{(j)}_k\sim \mathrm N_1(0,\sigma_k^2)$,
 and again using independence of $Y_k$ and  $\sum_{j=1}^k  \alpha_j Z^{(j)}_k$, we get $X_k\sim  \mathrm N_1(0,\sigma_k^2)$.
By \eqref{eq:th3lambda}, we have
  $$
  \var(X_n) =
 \var \left( \lambda \sum_{k=1}^{n-1}  \sigma_ k Y_k\right)+
 \var \left(d(1,\lambda)   \sum_{j=1}^{n-1}  \alpha_j Z^{(j)}_n\right)=\lambda^2  \sum_{i=1}^{n-1} \sigma_i^2
 +  (1-\lambda^2)  \sigma_{n-1}^2  =\sigma_n^2.
 $$
 Hence, $X_n\sim \mathrm N_1(0,\sigma_n^2)$.

 Next, we show that $(X_1,\dots,X_n)$ is a joint mix.
We can directly compute
\begin{align*}
\sum_{k=1}^{n} X_k &= \sum_{i=k}^{n-1} \lambda \sigma_k Y_k  +  d(1,\lambda)
\sum_{k=1}^{n-1}   \sum_{j=1}^k  \alpha_j Z^{(j)}_k
- \lambda  \sum_{k=1}^{n-1}   \sigma_k Y_k  +
 d(1,\lambda)
 \sum_{j=1}^{n-1}  \alpha_j Z^{(j)}_n
\\
&=  d(1,\lambda)
 \sum_{j=1}^{n-1}  \sum_{k=j}^{n}   \alpha_j Z^{(j)}_k
=0,
\end{align*}
where the last equality follows from the fact that
 $\mathbf Z^{(j)}$ is a joint mix  for each $j=1,\dots,n-1$.

 We check that $(X_1,\dots,X_n)$ is NA.
  This follows from the fact that $(X_1,\dots,X_n)$
  is the weighted sum of several independent
  NA random vectors
  $(\sigma_1Y_1,\dots,\sigma_{n-1} Y_{n-1}, -Y^*)$
  and $(\mathbf 0_{m-1},\mathbf Z^{(m)})$ for $m=1,\dots,n-1$.
 Alternatively, one can check that all non-zero terms in $\cov(X_i,X_j)$ are negative for $i\ne j$ 
 as follows.
 For $X_k, X_l$ with $k,l \le n-1$ and  $k < l$, we have 
\begin{align*}
\cov(X_k,X_l)&=\cov\left( \lambda \sigma_k Y_k+  d(1,\lambda)
 \sum_{j=1}^k  \alpha_j Z^{(j)}_k ,  \lambda \sigma_l Y_l+  d(1,\lambda)
 \sum_{i=1}^l  \alpha_i Z^{(i)}_l   \right)\\
 &=\lambda^2 \sigma_k  \sigma_l\cov(Y_k,Y_l)+ \lambda \sigma_k d(1,\lambda)
 \sum_{i=1}^l \alpha_i \cov(Z^{(i)}_l,Y_k) \\
 &~~~+\lambda \sigma_l d(1,\lambda)
 \sum_{j=1}^k  \alpha_j \cov(Z^{(j)}_k, Y_l)+d^2(1,\lambda)
 \sum_{j=1}^k \sum_{i=1}^l  \alpha_j\alpha_i\cov(Z^{(j)}_k,Z^{(i)}_l)\\
 &=-d^2(1,\lambda)
 \sum_{j=1}^k \frac{ \alpha^2_j}{n-j}\le0.
\end{align*}
For $X_k,X_n$ for all $k\le n-1$, we have
\begin{align*}
\cov(X_k,X_n)&=\cov\left( \lambda \sigma_k Y_k+  d(1,\lambda)
 \sum_{j=1}^k  \alpha_j Z^{(j)}_k ,  -\lambda \sum_{i=1}^{n-1}   \sigma_i Y_i  +
 d(1,\lambda)
 \sum_{i=1}^{n-1}  \alpha_i Z^{(i)}_n     \right)\\
 &=-\lambda^2 \sigma_k  \sum_{i=1}^{n-1} \sigma_i\cov(Y_k,Y_i)+ \lambda \sigma_k d(1,\lambda)
 \sum_{i=1}^{n-1} \alpha_i \cov(Z^{(i)}_n,Y_k) \\
 &~~~-\lambda  d(1,\lambda)
 \sum_{j=1}^k \sum_{i=1}^{n-1} \sigma_i\alpha_j \cov(Z^{(j)}_k, Y_i)+d^2(1,\lambda)
 \sum_{j=1}^k \sum_{i=1}^{n-1}  \alpha_j\alpha_i\cov(Z^{(j)}_k,Z^{(i)}_n)\\
 &=-\lambda^2\sigma_k^2-d^2(1,\lambda)
 \sum_{j=1}^k \frac{ \alpha^2_j}{n-j}\le0.
\end{align*}

Finally, the joint mix can be chosen as multivariate Gaussian by the construction of $(X_1,\dots,X_n)$ as the sum of Gaussian vectors.
\end{proof}

\begin{proof}[Proof of Theorem~\ref{th:r1-2}]
The ``if" statement is straightforward, and we will check the ``only if" statement.
Let $\XX=(X_1,\dots,X_n)$ be a joint mix and denote by $c=\sum_{i=1}^n X_i \in \R$. First, suppose that each component of $\XX$ is positive. 
 Denote by $V\subset \mathbb{R}$ the set of all possible values taken by random variables of the form $\sum_{i=1}^j X_i$ for $j=0,\dots,n$, with the convention that $\sum_{j=1}^{0} X_j=0$.
 Clearly, $V$ is finite. 
 The elements of $V$ are denoted by $v_0,v_1,\dots,v_K$ such that $v_0 < v_1<\dots<v_K$.
 Our assumptions imply that $v_0=0$, $v_1>0$ and $v_K=c$ because $\sum_{i=1}^n X_i=c$.
For $k \in [K]$ and $i\in [n]$, let
$$Y_{k,i}=   \id_{\{\sum_{j=1}^i X_j\ge v_k\}}-\id_{\{\sum_{j=1}^{i-1} X_j\ge v_k\}} $$
and let $\mathbf Y_k=(Y_{k,1},\dots,Y_{k,n})$.
Since each $X_j$ is positive, 
 the value of $ Y_{k,i}$ is either $0$ or $1$, and $$\sum_{i=1}^n Y_{k,i} = \id_{\{\sum_{j=1}^n X_j\ge v_k\}}-\id_{\{0\ge v_k\}}    = \id_{\{c\ge v_k\}}  =1 .$$ 
Therefore, $\mathbf Y_k$  follows a binary multinomial distribution for each $k \in [K]$.
Let $\mathbf X_k =(v_k- v_{k-1}) \mathbf Y_k$  for $k\in [K]$
with $v_0=0$.
Note that for $i\in [n]$,  
\begin{align*} 
\sum_{k=1}^K X_{k,i}
 & = \sum_{k=1}^K   (v_k-  v_{k-1})  \left ( \id_{\{\sum_{j=1}^i X_j\ge v_k\}}-\id_{\{\sum_{j=1}^{i-1} X_j\ge v_k\}} \right)
 = \sum_{j=1}^i X_j  - \sum_{j=1}^{i-1} X_j =
X_i,
\end{align*}
where we   used the identity 
$ \sum_{k=1}^K    (v_k-  v_{k-1})  \id_{\{x\ge v_k\}}   =x$ for $x\in V$. Therefore, $\sum_{k=1}^K  \mathbf X_k = \mathbf X$, showing   that $\mathbf X$ can be represented as a finite linear combination of binary multinomial random vectors.

If some components of $\mathbf X$ are not positive, we can take $m\in \R$ such that $X_i> m$ for each $i\in [n]$.
Applying the above result, we know that $(X_1-m,\dots,X_n-m)$ can be decomposed as the sum of JM joint mixes. 
Note that $\mathbf X=(X_1-m,\dots,X_n-m)+(m,\dots,m)$, and $(m,\dots,m)$ is $m$ times the sum of $n$ binary multinomial random vectors $(1,0,\dots,0), \dots, (0,\dots,0,1)$. Hence,
 $\mathbf X$  admits a finite linear combination of binary multinomial   random vectors.
\end{proof}

\begin{proof}[Proof of Proposition~\ref{prop:mixture}]
Let $G$ be the joint distribution of an exchangeable joint mix.
Let us write 
$$
\widetilde G (A)= \int_{\R^n} \delta_{\mathbf a}(A)  \d G(\mathbf a),~~~~A\in \mathcal B(\R^n),
$$  
where $\delta_{\mathbf a}$ is the point-mass at $\mathbf a$.
By exchangeability, we have $G( A^{\pi} ) = G(A)$ for $\pi\in {\mathfrak S_n}$ and $A\in \mathcal B(\R^n)$,
where $A^\pi$ is $\pi$ applied to elements of $A$. Therefore,
$$
G (A)= \int_{\R^n} \delta_{\mathbf a^\pi}(A) \d G(\mathbf a).
$$
Taking an average of the above formula over ${\mathfrak S_n}$, we have
$$
G (A)= \int_{\R^n}  U_{\mathbf a} (A) \d G(\mathbf a).
$$
It is known that each $U_{\mathbf a}$ is NA  \citep[][Theorem~2.11]{JP83}.
Moreover, the center of the joint mix distributed as $U_{\mathbf a}$ is $\mu$ since $G$ is supported on $\{(x_1,\dots,x_n)\in \R^n:x_1+\dots+x_n=\mu\}$.
\end{proof}

\begin{proof}[Proof of Proposition~\ref{prop:r1-1}]
As $\mathcal M$ is symmetric, we have $\sup_{\mu \in \mathcal M}C_{\mu}^f(\mathbf X)=\sup_{\mu \in \mathcal M}C_{\mu}^f(\mathbf X^{\pi})$ for all $\pi\in {\mathfrak S_n}$. Let $\Pi$ be   uniformly distributed on ${\mathfrak S_n}$ and  independent of $\mathbf X$. Plugging  $\mathbf X^\Pi$ in the objective \eqref{eq:mu}, we have
\begin{align*}
\sup_{\mu \in \mathcal M}C_{\mu}^f(\mathbf X^{\Pi})&=\sup_{\mu \in \mathcal M}\sum_{K\subseteq [n]} \E\left[f\left(\sum_{i\in K} X^{\Pi}_i\right)\right] \mu(K)\\
&=\sup_{\mu \in \mathcal M}\sum_{K\subseteq [n]}\frac{1}{n!} \sum_{\pi \in \mathfrak S_n}\E\left[f\left(\sum_{i\in K} X^{\pi}_i\right)\right] \mu(K)\\
&\le \frac{1}{n!}\sum_{\pi \in \mathfrak S_n}\sup_{\mu \in \mathcal M}\sum_{K\subseteq [n]}\E\left[f\left(\sum_{i\in K} X^{\pi}_i\right)\right] \mu(K)\\
&=\frac{1}{n!}\sum_{\pi \in \mathfrak S_n}\sup_{\mu \in \mathcal M}C_{\mu}^f(\mathbf X^{\pi})=\sup_{\mu \in \mathcal M}C_{\mu}^f(\mathbf X).
\end{align*}
Hence, $\sup_{\mu \in \mathcal M}C_{\mu}^f(\mathbf X^{\Pi})\le \sup_{\mu \in \mathcal M}C_{\mu}^f(\mathbf X)$. Furthermore, as $\mathbf X$ is a joint mix, we have  that $\mathbf X^\Pi$ is an exchangeable NCD joint mix with marginals $F$ and correlation matrix $P_n^*$.
\end{proof}

\begin{proof}[Proof of Theorem~\ref{prop:average:var}]

Without loss of generality, we assume that the variance of $F$ is 1. 
Using the same argument in the proof of Proposition \ref{prop:r1-1}, for any $\mathbf X$ with identical marginals $F$, we have 
$
\sup_{\mu \in \mathcal M}  C_{\mu}^2(\mathbf X^{\Pi}) 
\le 
\sup_{\mu \in \mathcal M}  C_{\mu}^2(\mathbf X)
$, where   $\Pi$ is  uniformly distributed on ${\mathfrak S_n}$.
 Let $\mathbf X_\rho$  be a random vector with  identical marginals $F$ and a 
correlation matrix whose 
off-diagonal entries are all $\rho$.
Since correlation matrices are positive semi-definite, 
 we have $\rho \in [-1/(n-1),1]$, with $\rho=-1/(n-1)$ attainable since $F$ is $n$-completely mixable. 
    The value of  $\sup_{\mu \in \mathcal M}  C_{\mu}^2(\mathbf X)$ only depends on the correlation matrix.  Therefore, it suffices to find an optimizer of the form $\mathbf X_\rho$ for some $\rho\in [-1/(n-1),1]$. 
Note that
\begin{align*}
 \sup_{\mu \in \mathcal M}C^2_{\mu}(\mathbf X_\rho)&=\sup_{\mu \in \mathcal M}\sum_{K\subseteq [n]} \left( \var\left(\sum_{i\in K} X_i\right)+\left(\E\left[\sum_{i\in K} X_i\right]\right)^2\right)\mu(K)\\
 &=\sup_{\mu \in \mathcal{M}} \sum_{k=1}^n \sum_{K \subseteq [n], \vert K\vert =k}\left(k+(k^2-k)\rho+k\E[X_1]\right)\mu(K).
 \end{align*}
It is clear that $\sup_{\mu \in \mathcal M}C^2_{\mu}(\mathbf X_\rho)$ increases in $\rho$. Therefore, the minimum is achieved at $\rho^*=-1/(n-1)$, which implies that $\mathbf X_{\rho^*}$ is an NCD joint mix with correlation matrix $P^*_n$. 
As the value of \eqref{eq:mu_uncertainty_quad} only depends on the correlation matrix, we have the desired result.
\end{proof}

\begin{proof}[Proof of Theorem~\ref{th:opt}]
The ``if" part is shown by  Theorem \ref{prop:average:var} by choosing $\mathcal M$ as both \eqref{eq:uncertainty_quad} and \eqref{eq:uncertainty2_quad} are special cases of \eqref{eq:mu_uncertainty_quad}. Next, we show the ``only if" part by showing that the correlation matrix of the minimizer to \eqref{eq:uncertainty_quad} or \eqref{eq:uncertainty2_quad} with any $k\in [n]\setminus\{1,n-1,n\}$  is $P_n^*$.

 Without loss of generality, we assume that the variance of $F$ is 1. As the mean of $F$ is zero, we have $\E[(\sum_{i\in K} X_i)^2]=\var(\sum_{i\in K} X_i)$.
By plugging an NCD joint mix $\mathbf X^{\rm E}$ with correlation matrix $P_n^*$ into  \eqref{eq:uncertainty_quad},   the optimal value for \eqref{eq:uncertainty_quad}   is 
  \begin{align} \max_{K\subseteq [n]}   \E\left[\left(\sum_{i\in K} X_i\right)^2\right]  = \max_{K\subseteq [n]}   \var\left(\sum_{i\in K} X_i^{\mathrm E}\right) =\max_{k \in [n]}\frac{k(n-k)}{n-1} = \frac{k^*(n-k^*)}{n-1},
 \label{eq:th2eq5p}
  \end{align}
  where $k^*=\lfloor n/2\rfloor$.

First, we consider the case $n=3$.
In this case, $ [n]\setminus \{1,n-1,n\}$ is empty,
  and we   only need to show that $P^*_n$ is the unique  correlation matrix of the minimizer  to \eqref{eq:uncertainty_quad}.
  Suppose that $\mathbf X$ with covariance matrix $\Sigma$ is a minimizer to \eqref{eq:uncertainty_quad}.
    By \eqref{eq:th2eq5p}, optimal value for \eqref{eq:uncertainty_quad} is $1$.
Hence,
$$
  \var\left(\sum_{i\in K} X_i\right)   \le 1  \mbox{  ~~~   for each $K$ with $|K|=2$,}
 $$
     and this implies
     \begin{equation}
     \sigma_{ij}\le -1/2,   \mbox{  ~~~   for $i\ne j$.}
\label{eq:th2eq5pp}
     \end{equation}
     Since $\Sigma$ is positive semi-definite,   we have $\sum_{i,j\in [3]} \sigma_{ij}\ge 0$, which leads to 
    $
    3+ 2\sigma_{12}+2\sigma_{13}+2\sigma_{23} \ge 0,
 $
     implying $ \sigma_{12}+ \sigma_{13}+ \sigma_{23} \ge -3/2.$
Together with
  \eqref{eq:th2eq5pp}, we get $ \sigma_{12} = \sigma_{13}= \sigma_{23} =-1/2$, and hence $\Sigma=P_n^{*}$.

Next, we consider the case $n\ge 4$.
 We first show that  the correlation matrix of the minimizer to \eqref{eq:uncertainty2_quad} is  unique for $1<k<n-1$.  Fix $k\in [n]\setminus \{1,n-1,n\}$.
  Suppose that $\mathbf X$ with covariance matrix $\Sigma$ is a minimizer to \eqref{eq:uncertainty2_quad}.
Our goal is to show $\Sigma=P_n^{*}$. 

Let $K_{\ell}$, $\ell=1,\dots, n_k$, be all subsets of $[n]$ with cardinality $k$, where $n_k= {n\choose k}$.  
Then we have  
\begin{equation}
 \frac 1 {n_k}  \sum_{\ell=1}^{n_k}  \var\left(\sum_{i\in K_\ell} X_i\right)   \ge k -\frac{k(k-1)}{ n-1} = \frac{k(n-k)}{n-1}.   \label{eq:th2eq3}   
\end{equation}
As $\mathbf X$ is a minimizer,  for each $K$ with $|K|=k$, we have
 \begin{equation}
\var\left(\sum_{i\in K} X_i\right)= \E\left[\left(\sum_{i\in K} X_i\right)^2\right] \le \max_{K\subseteq [n],~|K|=k}  \E\left[\left(\sum_{i\in K} X_i^{\rm E}\right)^2\right]   =  \frac{k(n-k)}{n-1}.
\label{eq:th2eq6}
     \end{equation}
Combining \eqref{eq:th2eq3}  and  \eqref{eq:th2eq6},
we have 
$$
 \var\left(\sum_{i\in K} X_i\right)    =  \frac{k(n-k)}{n-1} \mbox{  ~~~   for each $K$ with $|K|=k$.}
$$
Take $k=2$. For any $i,j \in [n]$ such that $i\neq j$, the above equation implies
$$\var(X_i+X_j)=\var(X_i)+\var(X_j)+2\cov(X_i,X_j)=\frac{2(n-2)}{n-1}.$$ As a result, we have $\cov(X_i,X_j)=-1/(n-1)$ for all $i,j \in [n]$ such that $i\neq j$. Hence, we conclude that $\Sigma =P_n^{*}$.

Finally, note that $k^*$ in  \eqref{eq:th2eq5p} satisfies $1<k^*<n-1$  for $n\ge 4$. We have justified that the correlation matrix for optimizers to \eqref{eq:uncertainty2_quad} with $k=k^*$
is unique. 
Therefore, by using \eqref{eq:th2eq5p},  the correlation matrix for optimizers to  \eqref{eq:uncertainty_quad} is also unique.

The above arguments show that, for $n \ge 3$, if $\mathbf X$ is a minimizer to \eqref{eq:uncertainty_quad}, then the correlation matrix of $\mathbf X$ is $P_n^*$, which implies that $\mathbf X$ is an NCD joint mix with correlation matrix $P_n^*$. The same conclusion holds true  if \eqref{eq:uncertainty_quad} is replaced by
  \eqref{eq:uncertainty2_quad} with any $k\in [n]\setminus\{1,n-1,n\}$.
\end{proof}

\begin{proof}[Proof of Proposition~\ref{prop:opt-n3}]
As the given marginal distributions have zero means, for any $(Y_1,\dots,Y_n)$ with variance vector $(\sigma_1^2,\dots,\sigma_n^2)$,
 $$
\max_{K\subseteq [n]}  \E\left[\left(\sum_{i\in K} Y_i\right)^2\right]=\max_{K\subseteq [n]}  \var\left(\sum_{i\in K} Y_i\right)
\ge \max_{i \in  [n]}  \var (Y_i)  = \max_{i\in [n]} \sigma_i^2.
 $$
In case $n=3$, a joint mix $\mathbf X$ with variance vector $(\sigma_1^2,\dots,\sigma_n^2)$ satisfies
 $$
\max_{K\subseteq [3],~|K|=1}  \var\left(\sum_{i\in K} X_i\right)   = \max_{K\subseteq [3],~|K|=2}  \var\left(\sum_{i\in K} X_i\right)
  = \max_{i \in  [3]}  \var (X_i) = \max_{i\in [3]} \sigma_i^2
,$$
and $\var(X_1+X_2+X_3)=0$.
Hence, the joint mix minimizes \eqref{eq:uncertainty_quad}.

To show that no positive covariance is allowed,
suppose that $(X_1,X_2,X_3)$ is a minimizer to  \eqref{eq:uncertainty_quad} and  $\cov(X_i,X_j)>0$ for some $i\ne j$. We  have
\begin{equation}\label{eq:no-pos}
\max_{K\subseteq [n]}  \E\left[\left(\sum_{i\in K} X_i\right)^2\right]
\ge  \E\left[\left(X_i+X_j\right)^2\right] = \var (X_i+X_j)  >\sigma_i^2+\sigma_j^2 \ge \max\left (\sigma_1^2,\sigma_2^2,\sigma_3^2\right),
\end{equation}
 where the last inequality follows from the necessary condition \eqref{eq:maxcond} of the existence of an NCD joint mix.
 Since we have seen that the optimal value of \eqref{eq:uncertainty_quad} is $\max_{i\in [3]} \sigma_i^2$,
\eqref{eq:no-pos} implies that $(X_1,X_2,X_3)$ does not minimize \eqref{eq:uncertainty_quad}.
\end{proof}

\begin{proof}[Proof of Lemma~\ref{lem:simple}]
One of the key properties of elliptical distributions is that they are closed under linear transformations, which is clear from \eqref{eq:ellip}.
Hence, for
$\mathbf X \sim \operatorname{E}_n(\boldsymbol{\mu}, \Sigma,\psi)$, the random variable
$\sum_{i=1}^n X_i$ follows $\operatorname{E}_n(\mathbf{1}_n^\top \boldsymbol{\mu}, \mathbf{1}_n^\top\Sigma\mathbf{1}_n,\psi)$, which is degenerate if and only if $\mathbf{1}_n^\top\Sigma\mathbf{1}_n=0$.
 Hence, $\mathbf X  $ is a joint mix if and only if $\mathbf{1}_n^\top\Sigma\mathbf{1}_n=0$ or $\psi =1$ on $\R_+$.
\end{proof}

\begin{proof}[Proof of Proposition~\ref{prop:R1-el}]
Necessity follows from Proposition \ref{prop:necessary}.
To show sufficiency, let $\mathbf{X}\sim \operatorname{E}_n(\boldsymbol{\mu},\Sigma,\psi)$, where $\Sigma$ is the dispersion matrix of the multivariate Gaussian distribution constructed in the proof of Theorem~\ref{th:nodfull}.
Then $X_i \sim  \operatorname{E}_1(\mu_i,\sigma_i,\psi)$, $i\in[n]$.
Moreover, it is checked in the proof of Theorem~\ref{th:nodfull} that $\mathbf{1}_n^\top\Sigma \mathbf{1}_n=0$ and $\sigma_{ij}\le 0$ for $i,j\in[n]$ such that $i\neq j$.
Therefore, $\mathbf{X}$ is the desired NCD joint mix.
\end{proof}

\begin{proof}[Proof of Theorem~\ref{th:impossible}]
The statement (\ref{item:ncd}) immediately follows from the facts that $\psi$ generates an $n$-dimensional elliptical distribution for every $n\in \N$ if and only if the corresponding elliptical class is a  Gaussian variance mixture family~\citep[Section~2.6]{FKN90}, and that $\operatorname{E}_n(\mathbf{0}_n,P_n^\ast,\psi)$ is an NCD joint mix, where $P_n^{*}$ is an $n\times n$ matrix whose diagonal entries are $1$ and off-diagonal entries are $-1/(n-1)$.

To show (\ref{item:nod}), we need two lemmas.

\begin{lemma}[Corollary~4 of~\cite{Y21}]\label{lem:elliptical:nod}
Let $\mathbf X \sim \operatorname{E}_n(\boldsymbol{\mu},\Sigma,\psi)$ and $\mathbf Y \sim \operatorname{E}_n({\boldsymbol{\mu}},\Sigma',\psi)$ be two elliptical distributions such that $\Sigma=(\sigma_{ij})$ and $\Sigma'=(\sigma_{ij}')$ satisfy $\sigma_{ii}=\sigma'_{ii}$ for all $i\in [n]$.
Then $\mathbf X \leq_{\rm cU}\mathbf Y$ if and only if 
$\sigma_{ij}\leq \sigma'_{ij}$ for all $i\neq j$.
\end{lemma}

Let $P_n^{\perp}$ be the identity matrix, which is the correlation matrix of an independent random vector.
Although Lemma~\ref{lem:elliptical:nod} implies that $\operatorname{E}_n({\boldsymbol{\mu}},P_n^{*},\psi)\leq_{\rm c}\operatorname{E}_n(\boldsymbol{\mu},P_n^{\perp},\psi)$ for general elliptical distributions, $\operatorname{E}_n({\boldsymbol{\mu}},P_n^{*},\psi) $ is not necessarily NOD in general since  $\operatorname{E}_n(\mathbf{0}_n,P_n^{\perp},\psi)$ does not have independent components.
In fact, an elliptical distribution $\operatorname{E}_n(\mathbf{0}_n,P_n^{\perp},\psi)$ is not NOD unless it is  Gaussian.

\begin{lemma}\label{lem:elliptical:zero:correlation}
 The  elliptical distribution $\operatorname{E}_n(\boldsymbol \mu,\Sigma,\psi)$ where $\Sigma$ is diagonal
 is not NOD unless it is Gaussian.
\end{lemma}

\begin{proof}
Assume that $\mathbf X \sim \operatorname{E}_n(\boldsymbol \mu,\Sigma,\psi)$ is NOD.
Since NOD is location invariant, it suffices to show the case when $\boldsymbol \mu=\mathbf{0}_n$.
When $\mathbf X$ is NOD, then so is $(X_1,X_2)$, that is,
$$
\p(X_1\leq x_1,X_2\leq x_2)\le \p(X_1\leq x_1)\p(X_2\leq x_2) \quad \text{for all }(x_1,x_2)\in \R^2.
$$
Since $(X_1,X_2)$ and $(-X_1,X_2)$ are identically distributed, we have
$$
\p(X_1\ge x_1,X_2\leq x_2)\le \p(X_1\ge x_1)\p(X_2\leq x_2) \quad \text{for all }(x_1,x_2)\in \R^2,
$$
and similarly, by symmetry,
$$
\p(X_1\ge x_1,X_2\ge x_2)\le \p(X_1\ge x_1)\p(X_2\ge  x_2) \quad \text{for all }(x_1,x_2)\in \R^2,
$$
$$
\p(X_1\le  x_1,X_2\ge  x_2)\le \p(X_1\le x_1)\p(X_2\ge x_2) \quad \text{for all }(x_1,x_2)\in \R^2.
$$
Adding the above four inequalities together, we get $1\le 1$. Hence, each of them is an equality.  However, $(X_1,X_2)$ follows a bivariate elliptical distribution with generator $\psi$, and thus $X_1$ and $X_2$ are not independent unless it is Gaussian;~see Theorem~4.11 of~\citet{FKN90}.
 Therefore, $\mathbf X$ cannot be NOD unless it is Gaussian. 
\end{proof}

Now we are ready to prove Theorem~\ref{th:impossible}.
The ``if" statement follows
  from Proposition \ref{prop:r1-G-1}.
  It remains to show the ``only if" statement.
Let $\psi$ be a characteristic generator different from that of the Gaussian distribution.
  For $n\ge 2 $, let $\mathbf X \sim  \operatorname{E}_n(\boldsymbol \mu, \Sigma,\psi)$ be an NOD joint mix
  where $\Sigma$ has positive diagonal entries.
  We start by observing from Lemma \ref{lem:elliptical:nod} that if $\sigma_{ij}>0$ for $i\ne j$, then the bivariate projection $(X_i,X_j) $ of $\mathbf X$ satisfies $(X_i,X_j)\ge_{\rm c} (X_i',X_j')$ where
$ (X_i',X_j')\sim \operatorname{E}_n(\boldsymbol \mu,\Sigma_{ij}',\psi)$
with $$\Sigma_{ij}'  =\begin{pmatrix} \sigma_{ii}&0 \\0&\sigma_{jj}\end{pmatrix}.$$
Using Lemma \ref{lem:elliptical:zero:correlation}, we know that $ (X_i',X_j')$ is not NOD, that is,
there exists $(x_i,x_j)\in \R^2$ such that
\begin{align}\label{eq:ellip:not:nod:1}
\p(X'_i\leq x_i,X'_j\leq x_j)> \p(X'_i\leq x_i)\p(X'_j\leq x_j);
\end{align}
note that it suffices to consider the inequality needed for NLOD (not NUOD) by symmetry of the elliptical distribution and location invariance of NOD.
Therefore, we have that
\begin{align}\label{eq:ellip:not:nod:2}
\p(X_i\leq x_i,X_j\leq x_j)> \p(X'_i\leq x_i,X'_j\leq x_j).
\end{align}
The two inequalities~\eqref{eq:ellip:not:nod:1} and~\eqref{eq:ellip:not:nod:2} imply that
$$
\p(X_i\leq x_i,X_j\leq x_j)> \p(X_i\leq x_i)\p(X_j\leq x_j),
$$
that is, $(X_i,X_j)$ is not NOD.
This leads to a contradiction.

Next, we assume $\sigma_{ij}\le 0$ for all $i\ne j$.
Since $\mathbf a^\top \Sigma \mathbf a\ge 0$ for all $\mathbf a\in \R^n$ and $\Sigma$ has positive diagonal entries, we can take $\mathbf a=(1/\sqrt{\sigma_{11}},\dots,1/\sqrt{\sigma_{nn}})$,
and this yields
$$
\sum_{i,j=1}^n \frac{\sigma_{ij}}{ \sqrt{\sigma_{ii} \sigma_{jj}} } = n+ \sum_{i\ne j} \frac{\sigma_{ij}}{ \sqrt{\sigma_{ii} \sigma_{jj}} }   \ge 0.
$$
Hence,   there exist $i,j$ with $i\ne j$ such that
 $$\rho_{ij}:=\frac{\sigma_{ij}}{ \sqrt{\sigma_{ii} \sigma_{jj}} } \ge -\frac{1}{n-1}.$$
Since NOD is location-scale invariant, the NOD of $(X_i,X_j)$ implies that $\mathrm E_2( \mathbf{0}_2, P_{ij}, \psi)$
is NOD, where
$$
P_{ij}  =\begin{pmatrix}1&\rho_{ij} \\\rho_{ij}&1\end{pmatrix}.
$$
Taking a limit as $n\to \infty$, 
and noting that NOD is closed under weak convergence \citep{MS02},
we conclude that
 $\operatorname{E}_2({\mathbf 0}_2,P_{2}^\perp,\psi)$ is also NOD,
 which contradicts Lemma~\ref{lem:elliptical:zero:correlation} if $\mathcal{E}(\psi)$ is not Gaussian.

Finally, by Proposition~\ref{prop:r1-G-1}, this characterization result follows if NOD is replaced by NSD or NA.
\end{proof}

\end{document}